\documentclass[11pt]{amsart}

\usepackage[margin=2cm]{geometry}

\usepackage{latexsym, amssymb, amsfonts, amsmath, amsthm, graphics, graphicx, subfigure, bbm, stmaryrd,float,caption}
\usepackage{palatino}
\usepackage{framed}

\newtheorem{theorem}{Theorem}[section]

\newtheorem{lemma}{Lemma}[section]

\newcommand{\R}{\mathbb{R}}

\numberwithin{equation}{section}

\title{Non-normal limiting distribution for optimal alignment scores of strings in binary alphabets}
\date{}
\author{Jun Tao Duan} 
\address{School of Mathematics\\
  Georgia Institute of Technology\\
  Atlanta, GA, 30332}
\email[Jun Tao Duan]{jt.duan@gatech.edu}

\author{Heinrich Matzinger} 
\address{School of Mathematics\\
  Georgia Institute of Technology\\
  Atlanta, GA, 30332}
\email[Heinrich Matzinger]{matzi@math.gatech.edu}

\author{Ionel Popescu} 
\address{School of Mathematics\\
  Georgia Institute of Technology\\
  Atlanta, GA, 30332 \\
  Institute of Mathematics of Romanian Academy\\ 
21 Calea Grivitei Street, 010702 Bucharest, Romania\\
Facultatea de Matematica si Informatica - Universitatea din Bucuresti \\
Str. Academiei nr.14, sector 1, C.P. 010014, Bucuresti, Romania}
\email[Ionel Popescu]{ipopescu@math.gatech.edu}

\begin{document}

\maketitle

\begin{abstract}
We consider two independent binary i.i.d. 
random strings $X$ and $Y$ of equal length
$n$ and the optimal alignments according to 
a symmetric scoring functions only. We decompose the
space of scoring functions
into five components. Two of these components add a part to the optimal
score which does not depend on the alignment and which is asymptotically normal.

We show that when we restrict the number of gaps sufficiently
and add them only into one sequence,
then the alignment score can be decomposed into a part which is normal and 
has order $O(\sqrt{n})$ and a part 
which is on a smaller order and tends to a Tracy-Widom distribution.
Adding gaps only into one sequence is equivalent to
aligning a string with its descendants in case of mutations and deletes.
For testing relatedness of strings, the normal part is irrelevant,
since it does not depend on the alignment hence it can be safely removed from the test statistic.

\end{abstract}

\section{Introduction}
 Optimal Alignments are
 among the main two components used by most DNA-alignment
algorithms which try to find similar substrings in DNA-sequences. 
Because of
its great   practical importance, these problems have received 
a tremendous amount of attention. 

Let us explain the key points from the DNA perspective. DNA sequencing is one of the most important areas of study in biology.  This is still rather expensive and very complex.  Suppose now that in a certain animal population, a certain gene is detected.  This gene is probably likely
to appear in the DNA of other animals too, but in a slightly different
form.  Thus,  instead of repeating the same costly experiments on other animals
to find the gene, they just search for a DNA-segment which looks similar
to the gene in the first population.  This is why often  DNA-strings in different animals
are compared in order to find places that are related.   For more details on the analysis of DNA we refer the reader to \cite{durbin1998biological} and in particular the iid assumption which appears also in this paper is explained on page 14 of the \cite{durbin1998biological}.  

{\footnotesize Here is a more concrete example.  Assume $Z=AATCUA$ is a hypothetical DNA of an ancestor with two descendants, one with the DNA string $X=AACA$ (which lost
the $T$ and the $G$) and another with the DNA equal to $Y=ATTCUA$ (which had an 
$A$ mutated into a $T$).  We align these strings (with gaps), as follows:
\[
\begin{array}{c|c|c|c|c|c|c|c}
Z& &A&A&T&C&U&A\\\hline
X& &A&A& &C& &A\\\hline
Y& &A&T&T&C&U&A
\end{array}
\]

In the analysis of the DNA one has access to the DNA of the descendants only.  In terms of the alignments, one is interested in the {\it historic alignment}. In our example,
we would have that the alignment $\pi$ is given by:
\begin{equation}
\label{alignment}
\begin{array}{c|c|c|c|c|c|c|c}
X& &A&A& &C& &A\\\hline
Y& &A&T&T&C&U&A
\end{array}
\end{equation}
and we call this the historic alignment.   The letters which disappeared in one descendant are now represented by gaps.  Which alignment with gaps corresponds to the evolution history of the 
strings is not known apriori. In other words, in our example, we know the strings $x=AACA$ and $Y=ATTCUA$, but not ``the true historic'' alignment given here by \eqref{alignment}.  

To guess which alignment is historic, the natural attempt is to find
an alignment which maximizes similarity. For this, one starts with a scoring function $S:\mathcal{A}\times\mathcal{A}\rightarrow \mathbb{R}$
which measures similarity between letters. (Here  $\mathcal{A}$ is the alphabet augmented by one symbol for the gap).  The {\it total score} of an alignment is then the sum of the scores
of the aligned symbol pairs. In alignment \eqref{alignment}, the alignment
score $S_\pi(x,y)$ is equal to 
\[
S_\pi(x,y)=S(A,A)+S(A,T)+S(G,T)+S(C,C)+S(G,U)+S(A,A)
\]
where $G$ stands for a gap.  An alignment $\nu$, which for given strings $x$ and $y$
maximizes the alignment score $S_\nu(x,y)$ is called {\it  an optimal alignment}.
The score of an optimal alignment is called \emph{optimal alignment score} and we will denote it by
\[
S(x,y)=\max_\pi S_\pi(x,y),
\]
where the maximum above is taken over all alignments with gaps
of $x$ and $y$.

Which alignment is optimal depends ultimately on the scoring function used.
Often, the scoring function biologists work with is a log likelihood.
For this, one assumes models of letters evolving independently of their neighboring letters. 
One takes as scoring function $S(A,B)$, the logarithm of the probability
that an ancestor letter evolves into an $A$ in one descendant and into
a $B$ in the other descendant. Here $A$ and $B$ are any two letters from
the alphabet under consideration, or also the gap symbol.
With this choice of scoring function, the optimal alignment
becomes the maximum-likelihood estimate for the historic alignment.
}

A high optimal alignment score indicates relatedness of the two DNA-strings.
Thus in testing for relatedness of DNA-strings one usually takes as test statistics
 the optimal alignment score.  The significance of such a test depends then on the order of magnitude of  the fluctuation of the optimal alignment score.

It has been a long open question whether for two independent i.i.d. strings
of length $n$,  the fluctuation of the optimal alignment score is of order 
$\sqrt{n}$ or $\sqrt[3]{n}$. (See \cite{Waterman-estimation}  and \cite{Sankoff1}).
In this paper, we think we can see why these two different conjectures
co-exist at the same time.   We show that for two independent
i.i.d strings of length $n$ the optimal alignment score (depending
on the parameters) can contain  a  normal component with fluctuation of order
$\sqrt{n}$. This component is not relevant for finding the optimal
alignment score, because it does not depend on the alignment.

\emph{  When testing for relatedness of DNA-strings one should remove this normal
component, since it is only noise. This can be done by
modifying the scoring function.    Once we ``cut
out this normal component'',  the question that remain is what is the asymptotic distribution
of the remainder?   Is it Tracy-Widom?  We can't prove it in full generality, however we check this for the special case where the number of gaps is a small power of the length of the strings and the gaps are allowed only into one string.}

\emph{ We also exhibit a scoring function with an alignment score which is of
the order $\sqrt{n}$ but not normal, nor Tracy-Widom.  This,
should then show that any claim that the  order $VAR[L_n(S)]=O(n)$
automatically implies asymptotic normality is wrong!
(Here $L_n(S)$ designates the optimal alignment score
of two i.i.d. random strings of length $n$.)
To prove normality, one  needs a more detailed knowledge of
the path structure of the optimal alignments! }

The optimal alignment can be viewed as a special case of Last Passage Percolation
with correlations.  To our knowledge, many of  the proofs for models of First and Last Passage Percolation,  where the asymptotic distribution is known, go back to proving
asymptotic equivalence with  the following functional
of Brownian motions:
Consider $B^1,B^2,\ldots$ a sequence of i.i.d. standard Brownian motions.  Also consider when $k$ goes to infinity, the functional
\begin{equation}
\label{shit}
\max_{0\leq t_1<t_2<\ldots<t_k\leq 1} [B^1(t_1)-B^1(0)]+[B^2(t_2)-B^2(t_1)]+\ldots+[B^k(t_k)-B^k(t_{k-1}))]
\end{equation}
which properly rescaled becomes Tracy-Widom \cite{MR1818248,MR1830441,MR1887169}.
Now, in our case, we have a similar situation, where instead we have a sequence of \emph{dependent} identically distributed
 standard Brownian motions  $W^1,W^2,\ldots$.  However,  the correlation structure is rather simple.  Precisely,  there exists a Brownian motion $V(t)$ and i.i.d Brownian motions $B^1,B^2,\ldots$ so that
\[
W^i(t)=B^i(t)+V(t),
\]
for all $i=1,2,\ldots$.  Hence, instead of independent Brownian motions we have correlated ones and the maximum given in \eqref{shit} can be written as
\begin{align*}
\label{shit2}
&\max_{0\leq t_1<t_2<\ldots<t_k\leq 1} [W^1(t_1)-W^1(0)]+[W^2(t_2)-W^2(t_1)]+\ldots+[W^k(t_k)-W^k(t_{k-1}))]=\\&=
V(1)-V(0)+\max_{0\leq t_1<t_2<\ldots<t_k\leq 1} [B^1(t_1)-B^1(0)]+[B^2(t_2)-B^2(t_1)]+\ldots+[B^k(t_k)-B^k(t_{k-1}))].
\end{align*}
Hence, we get a normal component plus a term on a smaller scale which is indeed Tracy-Widom. This is the reason why we can prove in the current  paper,  under a very special
situation (adding gaps only into one string and keeping their number a 
small power of $n$), that in the limit we obtain a normal plus Tracy-Widom on
a smaller scale.  However, the normal part is often irrelevant for
the optimal alignment problem.

Note here that the normal component $V(1)-V(0)$ does not dependent on the choice of 
$t_1<t_2<\ldots<t_k\leq1$ (which in what follows, correspond to alignments).  Again, the normal component is in some sense irrelevant for finding the optimal alignment. We should leave it out because it is only added noise.   

\subsection{History of this problem}
We consider two random i.i.d. strings $X=X_1X_2\ldots X_n$
and $Y=Y_1\ldots Y_n$.  In the sequel,  $L_n(S)$ denotes the optimal alignment score
of $X$ and $Y$. 

Chvatal-Sankoff proved  \cite{Sankoff1} that $E[L_n(S)]/n$ converges to a constant $\gamma_S$,
but $\gamma_S$ is not known in most cases.  In \cite{martinezlcs},
\cite{lcscurve} we have good confidence intervals for $\gamma$. We have been able to prove the 
order
\[
VAR[L_n(S)]=O(n)
\]
in many situations \cite{VARTheta},\cite{bonettolcs}, \cite{montecarlo},
\cite{increasinglcs}. This proves the conjecture of Waterman 
\cite{Waterman-estimation} 
(for certain specific distributions of the letters and  scoring 
functions),  but is different from what Chavatal-Sankoff have conjectured.
The order conjectured by Chavatal-Sankoff is the order
which one would expect knowing that $L_n$ is also  
the  weight of the heaviest path from $(0,0)$ to $(n,n)$ in an LPP-formulation
of our problem.
(In the physics literature, \cite{spohn91}, the order for the standard deviation
of the heaviest path from $(0,0)$ to $(n,n)$ in a vast class of 
LPP-models is conjectured to be $n^{1/3}$ not $n^{1/2}$.  For more details see also  \cite{KPZ86}).

Let us explain what LPP is and how our optimal alignment can be formulated
as an LPP-problem with correlated weights.  In LPP, one considers an oriented graph $(E,V)$
with a random weight function $w:E\rightarrow \mathbb{R}^+$ of the edges. 
Let $x,y\in V$. A path from $x$ to $y$ in the graph
$(E,V)$ is then defined as a sequence
of vertices: $x=x_0,x_1,x_2,\ldots,x_i=y$
where $x_j\in V$ and $(x_{j-1},x_j)\in E$ for all $j=2,3,\ldots,i$.
The  total weight of the  path $x=x_0,x_1,x_2,\ldots,x_i=y$
is given by:
$$w(\;(x_0,x_1,\ldots,x_i)\;):=\sum_{j=2}^i w(\;(x_{j-1},x_j)\;).$$
an optimal path from $x$ to $y$ is then a path
 which maximizes the
total weight among all paths from $x$ to $y$.\\

Now, our  optimal alignment problem can be formulated as an LPP on the integer lattice with the optimal alignment score $L_n(S)$ being the weight of the path of maximal weight from $(0,0)$ to $(n,n)$. For this we take the set of vertices $V$ to be $\mathbb{N}\times\mathbb{N}$
and the edges to go always one to the right, one up or diagonally
up to the next vertex. The weight for horizontal and vertical
edges is minus the gap penalty.  For the edge $((i-1,j-1),(i,j))$
the weight is $S(X_i,Y_j)$. With this setting when we align a letter
with a gap, this corresponds to moving one unit vertically or horizontally.
Aligning $X_i$ with $Y_j$  corresponds to moving along the edge
 $((i-1),(j-1)),(i,j))$. The optimal path then defines an optimal alignment where
for every edge $((i-1,j-1),(i,j))$ contained in the optimal path, aligns $X_i$ with $Y_j$.
The transversal fluctuation measures how far the optimal path from $(0,0)$ to $(n,n)$ deviates from the diagonal.\\
{\footnotesize
Let us give an example. Take the two strings $X=heini$ and $Y=henri$.
Consider the following alignment with gaps:
\begin{equation}
\label{alignmenthenri}
\begin{array}{c|c|c|c|c|c}
h&e&i&n& &i\\\hline
h&e& &n&r&i
\end{array}
\end{equation}
We can represent this alignment as a path
on the integer points in two dimension. We are always allowed to go
to the left by one unit or up by one unit or diagonal to the next integer point.
When we go horizontal or vertical, this means that we 
align a letter with a gap. When we go diagonal this means that
we align two letters with each other. Below, we show how we represent
the alignment given in \eqref{alignmenthenri} by a path on the two 
dimensional integer lattice: 
$$
\begin{array}{c||c|c|c|c|c}
i& & & & &6\\\hline
n& & &4&5& \\\hline
i& &3& & & \\\hline
e& &2& & & \\\hline
h&1& & & & \\\hline\hline
0&h&e&n&r&i
\end{array}
$$
To visualize, the path represented by the last table above, draw short
segments between the points: $0$ to $1$, then $1$ to $2$,
....and up to a segment from $5$ to $6$. The path is then given 
by this sequence of segments. 
So, the first such segment from $0$ to $1$ is diagonal, and corresponds
to aligning the two $h$'s with each other. The score-contribution
for this is $S(h,h)$. Then, the segment from $1$ to $2$ is diagonal
and corresponds to aligning the $e$'s. Its score contribution
is $S(e,e)$. The segment from $2$ to $3$ is vertical and hence corresponds
to aligning a letter with a gap. In the current case, it is the first
$i$ in $heini$ which gets aligned with a gap, and the score
for this is the gap penalty $\delta<0$. Then $34$ is again diagonal
and corresponds to aligning $n$ with $n$, with score $S(n,n)$.
Then, the segment $45$ is horizontal and corresponds to aligning
$r$ in $henri$ with a gap, which contributes to
minus a gap penalty to the total score.  The final segment $56$ is diagonal,
and corresponds to aligning the last $i$ of $henri$ with the last $i$
of $heini$. Score is then $S(i,i)$.
In other words, the horizontal or vertical moves contribute minus a gap 
penalty each, whilst the diagonal moves contribute to the score by aligning
the corresponding letters.  IN THIS EXAMPLE, WE ALIGNED ONLY IDENTICAL
LETTERS. This is not always the case. In DNA-analysis, when letters can mutate
we also align similar letters with each other, not just identical.

Formally we can view an alignment with gaps of the strings $x=x_1x_2\ldots x_n$
and $y=y_1y_2\ldots y_n$ as a couple consisting of two
increasing integer sequences. The sequence $\nu$
\[
1\leq \nu_1<\nu_2<\ldots<\nu_i\leq n
\]
 and the sequence $\mu$
\[
1\leq \mu_1<\mu_2<\ldots<\mu_i\leq n,
\]
are non-negative integers.  Here $i\leq n$ is any integer non-negative number. The alignment with gaps
defined by $(\mu,\nu)$ is then the alignment which aligns
$x_{\nu_j}$ with $y_{\nu_j}$ for all $j=1,2,\ldots,i$ and aligns all remaining
 letters 
with a gap. The alignment score
is then 
$$\sum_{j=1}^i S(x_{\nu_j},y_{\mu_j})-2(n-i)\cdot \delta,$$
 where is $2(n-i)\delta$ represents the total gap penalty.
(The penalty for one given gap being $-\delta$).
In our current example, we have the number of aligned letter pairs
is $j=4$. Furthermore, we find:
 $\nu_1=1,\nu_2=2,\nu_3=4,\nu_4=5$, whilst $\mu_1=1,\mu_2=2,\mu_3=3,\nu_4=5$.
The path representing the alignment given by $(\nu,\mu)$ is the path 
containing all the segments:
$$
((\nu_j-1,\mu_j-1),(\nu_j,\mu_j))$$
for $j=1,2,\ldots,i$
and additionally a minimum amount of vertical and horizontal
unit length integer segments so as to make this a connected path.
}\\[3mm]

First Passage Percolation (FPP) and LPP are part of a vast area of 
statistical physics \cite{spohn91} which is concerned with random growth models
for which physicists expect some universality properties.
More specifically, one considers growth of a cluster where
material is being attached randomly on the surface of a nucleus. 
There are many fundamental questions open for decades, such as the universality of the 
fluctuation exponents.  Many theoretical physicists, 
crystallographers and probabilists have worked on this and many conjectures have been formulated \cite{spohn91}, as for example that the fluctuation should behave like $n^{1/3}$ and the transversal fluctuation should be or order $n^{2/3}$ (see KPZ-conjecture in \cite{KPZ86}).  These are difficult conjectures and rigorous results were obtained only under some special circumstances of LPP models, like Longest Increasing Subsequence \cite{BaikDeiftJohansson99}, 
LPP  on  $\mathbb{Z}^2$ with exponential or geometric waiting times 
\cite{Johansson2000} and the case of percolation in a narrow tube 
\cite{chatterjee2012}.  To our knowledge, the best results show only
a fluctuations of less than $n/\ln(n)$ \cite{benjamini2003} and above $\ln(n)$
\cite{peres1994}.

We have been able to prove our results for several  Longest Common
Subsequence (LCS) and Optimal Alignments (OA) models.\\
{\footnotesize LCS are a special case of OA. For two strings
$x=x_1x_2\ldots x_n$ and $y=y_1y_2\ldots Y_n$ a common subsequence
of $x$ and $y$ is a sequence which is a subsequence of $x$ and at the same
time a subsequence of $y$. A LCS of $x$ and $y$ is a common subsequence
of $x$ and $y$ of maximum length. The length of the LCS of $x$ and $y$
can be viewed as an optimal alignment score of $x$ and $y$.
For this we take $S(a,b)$ to be $1$ if $a=b$ and $S(a,b)=$ otherwise.
On top the gap penalty should be taken $0$. With this choice,
the length of the LCS is also equal to the optimal alignment score}.\\[3mm]
Our results together
with the results of Baik, Deift and Johansson \cite{BaikDeiftJohansson99},
are among the few  rigorous results determining the exact order of the fluctuation exponents for some non-trivial FPP or LPP related 2-D models.  In their work, as in ours, the fluctuation is of order $n^{1/3}$ and the rescaled asymptotic limiting law is Tracy-Widom. Another interesting paper is for the TASEP
model,
by Borodin, Ferrari, Praehofer and Sasamoto \cite{TASEP}. There they do not use
the same approach of reducing the problem to formula \eqref{I} and \eqref{II},
and it is not  the GUE which appears in the limit, but GOE. Other methods used
to establish the solvability of LPP models are the vertex operator and fock space formalism put foward by Okounkov \cite{okounkov} in the study of Schur process.  

Our paper is organized as follows.  In Section~\ref{s:2}, we introduce the main decomposition of the scoring function and show examples where the limiting distribution of the optimal alignment score  is normal alongside with other examples which are non-normal.   In the next section, namely Section~\ref{s:3}, we restrict the number of gaps to a certain number $k$ which is finite and fixed and allow the gaps only in one sequence.   Here we give a general description of the optimal alignment in terms of multiple dependent random walks and outline an argument which supports the conjecture that under some conditions, if $k$ is replaced by $n^{\alpha}$ with small $\alpha$, then the (rescaled) limiting distribution of the optimal alignment is Tracy-Widom.   We can not prove this statement rigorously in general, but in the next section, namely Section~\ref{s:4}, we prove this for the key scoring function $S(x,y)=xy$ under the assumption that the number of gaps is $k=n^{\alpha}$  for small $\alpha$ and the second sequence of letters is distributed normally rather than $\pm1$.   The main rigorous result is stated in Theorem~\ref{t:1} and appears in Section~\ref{s:4}.     

Section~\ref{s:5} is dedicated to technical results while the last one, Section~\ref{s:6} contains a few simulations.

\section{Linearly decomposing the space of scoring functions}\label{s:2}

In this section we discuss the general decomposition of the scoring function in a natural basis.  The point of this is that some of the scoring functions are in some sense artificial because the optimal score does not depend on the alignment.  This will give a normal behavior in the limit for the optimal score.   

We consider symmetric scoring for a binary alphabet.
We do not allow alignments of gaps with gaps, so
the scoring functions we consider are given by
a $3\times 3$ matrix, with one irrelevant entry, namely 
$$S:=
\left(
\begin{array}{ccc}
S(a,a)&S(a,b)&S(a,G)\\
S(b,a)&S(b,b)&S(b,G)\\
S(G,a)&S(G,b)&S(G,G)
\end{array}
\right).$$
Here $G$ stands for a gap. In our case we will put
$S(G,G)=0$ and we will always assume that the above
matrix is symmetric hence $S(a,b)=S(b,a)$
and $S(a,G)=S(G,a)$, whilst $S(b,G)=S(G,b)$.
The linear space of such $3\times 3$ symmetric matrices
is generated by the following five elements

$$S_0:=\left(
\begin{array}{rrr}
1&1&1/2\\
1&1&1/2\\
1/2&1/2&0
\end{array}  \right),\,S_1=\left(
\begin{array}{rrr}
1&0&1/2\\
0&-1&-1/2\\
1/2&-1/2&0
\end{array}  \right),\,S_2=
\left(
\begin{array}{rrr}
1&-1&0\\
-1&1&0\\
0&0&0
\end{array}  \right)
$$

and
$$S_3=
\left(
\begin{array}{rrr}
0&0&1\\
0&0&-1\\
1&-1& 0
\end{array}  \right),\,
S_4=\left(
\begin{array}{rrr}
0&0&1\\
0&0&1\\
1&1&0
\end{array}  \right).
$$
Note that the matrices $S_0$, $S_1$ and $S_2$ are orthogonal to each other (with respect to the canonical inner product given by $\langle A,B \rangle=Tr(AB)$). Also, the matrices $S_3$ and $S_4$ concern  only with the gap penalty. Thus, any scoring function we will consider
can be written as
\[
S=a_0S_0+a_1S_1+a_2S_2+a_3S_3+a_4S_4.
\]

First, it is  important to note that the part $a_0S_0+a_1S_1$
gives the same alignment score for any alignment.
This is so, because $S_0$ and $S_1$ can both be viewed as linear additive
functions in the following way.   For a given a function 
$h:\{a,b,G\}\rightarrow \mathbb{R}$, take the ``linear ``scoring function $T$ defined by:
\begin{equation}
\label{h}
T(x,y):=h(x)+h(y)
\end{equation}
for all $x,y\in\{a,b,G\}$.  

We also assume that $h(G)=0$. Clearly for such a ``linear'' scoring function,
with value $h(G)=0$, the alignment score for any alignment $\pi$
does not depend $\pi$.  Indeed, it is trivial to check that 
\[
T_\pi(X_1\ldots X_n,Y_1\ldots Y_n)=\sum_{i=1}^n h(X_i)+\sum_{i=1}^n h(Y_i).
\]
Furthermore, the expression on the right side of the above equality
is a linear function of a binomial random variable.
Now, in our case, $S_1$ is of the above type with the function 
$h$ given by 
$$h(a)=1/2,h(b)=-1/2,h(G)=0.$$
Thus, we find that
\[
L_n(S_1)=(N_n^a-N_n^b)/2
\]
where $N_n^a$ denotes the total number of $a$'s in both strings
$X$ and $Y$ combined whilst $N_n^b$ denotes the number
of $b$'s in both strings combined. Noting that $N_n^a+N_n^b=2n$,
we finally find that
\[
L_n(S_1)=N_n^a-n.  
\]
Here $N_n^a$ is a binomial with parameter $2n$ and $p_{a}=P(X_1=a)$, consequently,  
\[
VAR[L_n(S_1)]=VAR[N_n^a]=2np_{a}(1-p_{a}).
\]

Similarly, the alignment score according to $S_0$ does not depend
on the alignment and is also equal to $n$.  Since the alignment scores according to $S_0$ and $S_1$ do not depend on  the alignment we find for any
scoring function $S=a_0S_0+a_1S_1+a_2S_2+a_3S_3+a_4S_4$
that  the optimal alignment score can be decomposed as follows: 
\begin{align*}
&L_n(S)=L_n(a_0S_0+a_1S_1+a_2S_2+a_3S_3+a_4S_4)=\\
&=a_0L_n(S_0)+a_1L_n(S_1)+L_n(a_2S_2+a_3S_3+a_4S_4)=\\
&(a_0-a_1)n+a_1\cdot N_n^a+L_n(a_2S_2+a_3S_3+a_4S_4).
\end{align*}

According to the central limit theorem, $N_n^a$
is asymptotically normal on the scale $\sqrt{n}$ and thus we can conclude that 
\[
\frac{N_n^a-2np_a}{\sqrt{2n p_a (1-p_a)}}
\approx \mathcal{N}(0,1),
\]
where $p_a=P(X_1=a)$.

Now, since the part $L_n(a_0S_0+a_1S_1)$ does not depend
on the alignment it has no affect on which alignment is optimal.
Also, it is just noise added for testing if sequences are related
or not.   This noise can be removed, since it is known.  Equivalently,  we can just take $a_2S_2+a_3S_3+a_4S_4$
as scoring function instead of $S=a_0S_0+a_1S_1+a_2S_2+a_3S_3+a_4S_4$.

\subsection{Normal and non-normal cases}

In this section we present a simple example for which the optimal alignment score, properly rescaled, converges to a distribution which is not normal.  Interestingly, the scaling is still the same as in the case of the previous section, namely, $\sqrt{n}$.  

As mentioned, for scoring functions of the type $S=a_0S_0+a_1S_1$,
the alignment score does not depend on the alignment $\pi$
of the strings $X_1\ldots X_n$ and $Y_1\ldots Y_n$.
Furthermore, the limiting distribution of the optimal
alignment score is normal on the scale $\sqrt{n}$
since

\[ 
L_n(a_0S_0+a_1S_1)=a_0 n+a_1(N^a_n-n)
\]
where $N^a_n$ is the total number of $a$'s in the 
concatenated string
$X_1\ldots X_nY_1Y_2\ldots Y_n$.   Since $N^a_n$ is binomial with parameter $2n$ and $p_a=p(X_1=a)$,  the above optimal score, properly rescaled converges to the normal.  
Thus we can interpret the scoring function
with $a_2=a_3=a_4=0$ as degenerate and hence not relevant
for practical purposes.

Next we consider the following scoring function
\[
S=\left(
\begin{array}{rrr}
1&0&0\\
0&-1&0\\
0&0&0
\end{array}  \right)=S_1- S_3/2.
\]
If $N_a^X$ designate the number of $a$'s in $X=X_1X_2\ldots X_n$
and $N_a^Y$ designate the number of $a$'s in $Y=Y_1\ldots Y_n$, 
then, the optimal alignment score according to the scoring function
$S$ is simply equal to the minimum of $N_a^X$ and $N^Y_a$, namely,
\[
L_n(S)=\min\{N_a^X,N_a^Y\}.
\]
This is so because aligning a $b$ with a $b$ gives a negative
score, whilst aligning something with a gap gives $0$.
So, the optimal alignment is going to align all $b$'s with gaps.
Since an $a$ aligned with an $a$ scores positively, the 
optimal alignment is going to align a maximum number of $a$'s
with each other.\\
Note that $N_a^X$ and $N_a^Y$ have the same expectation
and both tend asymptotically to be normal when re-scaled
by $\sqrt{np(1-p)}$.
Hence, we get
$$\frac{L_n(S)-E[N_a^X]}{\sqrt{np(1-p)}}=
\min\left\{\frac{N_a^X-E[N_a^X]}{\sqrt{np(1-p)}},
\frac{N_a^Y-E[N_a^Y]}{\sqrt{np(1-p)}}\right\}
\approx\min\{\mathcal{N}_1,\mathcal{N}_2 \}
$$
where $\mathcal{N}_1$ and $\mathcal{N}_2$
are two standard normal independent of each other, which is certainly not normal.   Thus,  we have a non-normal limiting distribution, but still we have a fluctuation on a scale
$\sqrt{n}$! 

Notice also the following interesting fact, namely that there are many different alignments, which are very far away from each other.  Indeed, 
say that $N_a^X>N_a^Y$. Then, the optimal alignments can be characterized  
as follows.  Choose a subset of $a$'s in $X$
with cardinality equal  to $N_a^X-N_a^Y$. These $a$'s are aligned
with gaps.  Then align all the remaining $a$'s in $X$ with $a$'s from $Y$.
Finally, align all the $b$'s with gaps. 

%The point is that in this case there are many optimal alignments which could be very different from each other.   
%If you represent the set
%of all such optimal alignments in a two dimensional diagram,
%you will get something like a sausage of diameter $\sqrt{n}$.....(TO BE
%VERIFIED) and hence you have way more degrees of freedom than
%in most situations of optimal alignment scores where you have
%many unicity points......so something which on macroscopic
%level is different from the non-degenerate situations.

Now, real life scoring functions should not have $a_4$ too small,
since in practical situations, like genetic evolution, there are usually
not too many gaps.   This reflects the fact that there are not too many 
deletes from an ancestors DNA.   In the next section we show that if we look for the optimal alignment under the constrain of the total number of gaps being small,
then we get that $L_n(a_2S_2+a_3S_3+a_4S_4)$ can be decomposed
into a normal part on the scale $\sqrt{n}$ plus a smaller part whose limiting
distribution is non-normal.

\section{Fixing the number of gaps and allowing them only in one
sequence}\label{s:3}

In this section, under the assumptions that the number of gaps is finite and fixed, we present a general outline of an argument which should be sufficiently convincing that the limiting behavior of the optimal alignment is normal in the scale $\sqrt{n}$ plus a Tracy-Widom on a smaller scale.   

Let us assume that we only consider alignments with a fixed number of gaps
equal to $k$, where $k$ is a fixed constant. To simplify the problem
we will allow the gaps only in one sequence.

Let $k$ be a constant and let $X=X_1X_2\ldots X_{n-k}$
and $Y=Y_1Y_2\ldots Y_n$ be the two random strings of letters. Thus, we will align exactly $k$ letters  from $Y$ with gaps, but all the letters of $X$ will
be aligned with letters of $Y$. We will consider the optimal alignment
of $X$ and $Y$ under the constrain that there are exactly $k$ gaps all inserted in $X_{1},\dots X_{n-k}$. The corresponding optimal score
will be denoted by $L_n^k(S)$. To specify an alignment we simply
need to specify which $k$ letters of $Y$ get aligned with gaps, thus we need to specify a sequence of $k$ increasing integers
$0<c_1<c_2<\ldots<c_k\leq n$. The vectors consisting of these integers
then define our alignment, namely, 
align $Y_{c_i}$ with a gap for all $i=1,2,3,\ldots,k$.

Let $R^i$ denote the random walk defined recurrently on $j$ by
\begin{align*}
&R^i(0)=0\\
&R^i(1)=S(X_{1-i+1},Y_1)
\end{align*}
and then:
$$R^i(j)-R^i(j-1)=S(X_{j-i+1},Y_j)
$$
for all $i=1,2,\ldots,k+1$ and for $j\in[1,n-1]$. 
(To simplify the notation we assume that the sequence of the $X_i$'s
is a double infinite sequence of i.i.d. variables 
$$\ldots,X_{-2},X_{-1},X_0,X_{1},X_{2},\ldots.$$
In the problem which we will consider only a finite number
of $X_i$'s will have $i\leq 0$ and their contribution
will be negligible).

{\footnotesize For example, the first random walk $R^1$ is then defined
by
\begin{align*}
&R^1(0)=0,\\&R^1(1)=S(X_1,Y_1),\\&R^1(2)=S(X_1,Y_1)+S(X_2,Y_2),\\&
R^1(3)=S(X_1,Y_1)+S(X_2,Y_2)+S(X_3,Y_3),\\&\ldots
\end{align*}
Similary, the second random walk $R^2$ is defined as follows:
\begin{align*}
&R^2(0)=0,\\&R^2(1)=S(X_0,Y_1),\\&R^1(2)=S(X_0,Y_1)+S(X_1,Y_2),\\&
R^2(3)=S(X_0,Y_1)+S(X_1,Y_2)+S(X_2,Y_3),\\&\ldots
\end{align*}}\\[3mm]
Consider the  alignment score under the constrain that 
the $l$-th gap is aligned with $Y_{c_l}$, for all $l=1,2,\ldots,k$, where
$0<c_1<\ldots<c_k\leq n$ is a given increasing sequence of integers.
The alignment with exactly $k$ gaps is uniquely defined
by this constrain and its alignment score is:
$$ (R^1(c_1-1)-R^1(0))+(R^2(c_2-1)-R^2(c_1))+\ldots +(R^{k+1}(n)-R^{k+1}(c_k)).
$$
Recall that we assume a zero gap penalty.

{\footnotesize Let us give an example.
Take the string $Y=Y_1Y_2Y_3Y_4Y_5Y_6Y_7Y_8$ and $X=X_1X_2X_3X_4X_5X_6$
we could align $Y_3$ and $Y_{6}$ with a gap. Then we have $k=2$ gaps
and $c_1=3$ whilst $c_2=6$. So, this defines the  $2$-gap-alignment $\pi$
given by:
$$
\begin{array}{c|c|c|c|c|c|c|c}
X_1&X_2&   &X_3&X_4&  &X_5&X_6\\\hline
Y_1&Y_2&Y_3&Y_4&Y_5&Y_6&Y_7&Y_8
\end{array}
$$
Hence,  the alignment score with a $0$-gap penalty is equal to:
\begin{align*}
\tt{Alignment \; score\;of\;} \pi=&S(X_1,Y_1)+S(X_2,Y_2)+\\
&S(X_3,Y_4)+S(X_4,Y_5))+\\
&S(X_5,Y_7)+S(X_6,Y_8)\\
= &R^1(2)-R^1(0)+\\
&R^2(5)-R^2(3)+\\
&R^3(8)-R^3(6).
\end{align*}
}\\[3mm]
Let us denote by 
$L_n^k(S)$ the optimal alignment score under the constrain that the 
alignment contains exactly $k$ gaps all aligned with letters from $Y$.  Since the gap penalty is $0$ the formula for
$L_n^k(S)$  is then given by
\begin{equation}
\begin{split}
L_n^k(S):=
\max_{0=c_0<c_1<\ldots<c_k<n} (R^1(c_1-1)-R^1(0))&+
(R^2(c_2-1)-R^2(c_1))+(R^3(c_3-1)-R^3(c_2))+  \\ 
&\dots+(R^{k+1}(n)-R^{k+1}(c_k)).
\end{split}
\end{equation}
Hence, the maximum is obtained by choosing  the increasing sequence
$0<c_1<c_2<\ldots<c_k<n$ which maximizes
the following sum:\\
take the increment of the first
random walk on $[0,c_1]$, then we add the increment of the second random
walk on $[c_1,c_2-1]$ and keep on going until the increment of the last random walk
on $[c_k,n]$ has been added.  

Next, we are going to put the different one-dimensional random walks $R^i$
together into one vector: 
$$\vec{R}(t):= \left(\begin{array}{l}
R^1(t)\\
R^2(t)\\
R^3(t)\\
\ldots\\
R^{k+1}(t)
\end{array}\right)
$$
In this manner $\vec{R}$ becomes a $k+1$-dimensional "correlated" random walk.
Thus we have
$\vec{R}(0)=(0,0,0,\ldots,0)^T$
and
\[
\vec{R}(1)=\left(
\begin{array}{l}
S(X_1,Y_1)\\
S(X_0,Y_1)\\
S(X_{-1},Y_1)\\
\ldots\\
S(X_{1-k},Y_1)
\end{array}\right)
,
\vec{R}(2)-\vec{R}(1)=\left(
\begin{array}{l}
S(X_2,Y_2)\\
S(X_1,Y_2)\\
S(X_{0},Y_2)\\
\ldots\\
S(X_{2-k},Y_2)
\end{array}\right)
,\ldots .
\]

If we subtract the drift and rescale by $\sqrt{n}$, the multidimensional correlated random walk $\vec{R}(t)$ converges to a $k+1$ dimensional Brownian motion and we refer to \cite{MR0375412,MR0402883} for more details. This means, that holding
$k$ fixed, we can find a $\delta_n$ which goes to $0$
as $n$ goes to infinity so that
there exists a $k+1$-dimensional Brownian motion 
$\vec{B}(t)=(B^1(t),B^2(t),\ldots,B^{k+1}(t))^T$
for which
\begin{equation}
\label{approxB}
\max_{t\in[0,1]} \left|
\frac{\vec{R}(tn)-E[\vec{R}(tn)]}{\sqrt{n}}-\vec{B}(t)\right|\leq 
\delta_n,
\end{equation}
holds with probability $1-\epsilon_n$
where $\epsilon_n$ goes to $0$ as $n\rightarrow\infty$.  Here the Brownian motion $\vec{B}(t)$ depends on $n$, but of course,  its distribution does not. 
  If $s\in[0,n]$ is not an integer
then we take for $\vec{R}(s)$
 the linear interpolation of $k\mapsto \vec{R}(k)$.
Let us introduce a small modification of $L_n^k(S)$.
This modification will be denoted by
$\tilde{L}_n^k(S)$ and is equal
to
\begin{align*}
&\tilde{L}_n^k(S)=\\
&\max_{0<c_0<c_1<\ldots<c_k<n)} (R^1(c_1)-R^1(0))+
(R^1(c_2)-R^1(c_1))+(R^2(c_3)-R^2(c_2))+\ldots +(R^{k+1}(n)-R^{k+1}(c_k)).
\end{align*}
Clearly with this definition,
\begin{equation}
\label{tilde}
|\tilde{L}_n^k(S)-L_n^k(S)|\leq C\cdot k,
\end{equation}
where $C:=\max_{a,b\in\mathcal{A}}|S(a,b)|.$
Note that for all the random walks $R^i$, the drift is the same and equal 
to $E[S(X_1,Y_1)]$.
Now, we can rewrite $\tilde{L}^k_n(S)$ using
the unbiased random walk
\begin{equation}
\label{tilde2}
\tilde{L}_n^k(S)=
\left(\max_{0<c_0<c_1<\ldots<c_k<n} \sum_{i=1}
\left((R^i(c_i)-R^i(c_{i-1}))-E[R^i(c_i)-R^i(c_{i-1})]
\right)\right)\;\;+\;\;n\cdot E[S(X_1,Y_1)].
\end{equation}
Assume next that inequality \eqref{approxB} holds.
Then, we can approximate the unbiased random walks 
 on the right side of \eqref{tilde2}
 by  components of a Brownian motion according to  \eqref{approxB}.
Since the sum contains $2(k+1)$ terms, the approximation error 
will be at most $\delta_n\cdot2(k+1)$ (if we assume
\eqref{approxB} to hold) and thus
\begin{equation}
\label{gissv}
\left|\frac{\tilde{L}^k_n(S)-n\cdot E[S(X_1,Y_1)]}{\sqrt{n}}-L^k(\vec{B})
\right|\leq 
2\delta_n(k+1),
\end{equation}
where $L^k(\vec{B})$ refers to the maximum
$$L^k(\vec{B}):=\max_{0<t_1<\ldots<t_k<1}\left((B^1(t_1)-B^1(0))+
(B^2(t_2)-B^2(t_1))+\ldots+(B^{k+1}(1)-B^{k+1}(t_k))\right).
$$
We can now use inequality \eqref{gissv} together with \eqref{tilde}
to find
\begin{equation}
\label{gissv2}
\left|\frac{L^k_n(S)-n\cdot E[S(X_1,Y_1)]}{\sqrt{n}}-L^k(\vec{B})
\right|\leq 
2\delta_n(k+1)+ \frac{2C\cdot k}{\sqrt{n}}.
\end{equation}

The component of our Brownian motions $\vec{B}(t)$ could be
correlated.  To see this, let's calculate the covariance between the first two random walks,
that is between $R^1(n)-R^1(0)$ and $R^2(n)-R^2(0)$.
The idea is that most terms of the sum corresponding to the random
walk are not correlated, however the closer ones are correlated. To the point, take for example
the first term $S(X_{1},Y_{1})$ in the sum
\begin{equation}
\label{R1}
R^1(n)-R^1(0)=S(X_1,Y_1)+S(X_2,Y_2)+\ldots+S(X_n,Y_n).
\end{equation}
Since, the $X_i$'s and $Y_i$'s are
all independent of each other,  an expression of the type $S(X_i,Y_j)$ can only be correlated
with $S(X_1,Y_1)$ if either $i=1$ or $Y=1$.  Therefore,  in the sum
\begin{equation}
\label{R2}
R^2(n)-R^2(0)=S(X_0,Y_1)+S(X_1,Y_2)+\ldots+S(X_n,Y_n)
\end{equation}
there are only two terms which can have non-zero correlation
to $R^1(1)-R^1(0)=S(X_1,Y_1)$. These terms are $S(X_0,Y_1)$ and $S(X_1,Y_2)$.
This analysis holds true for every term in the sum for the first random walk. 
Only in the very last term in the sum on the right side of \eqref{R1} there is a term
for which there are not two terms in the sum on the right side
of \eqref{R2} correlated with it. Hence, the covariance
between \eqref{R1} and \eqref{R2} is about $2n$ times
constant, where the constant is given as the covariance
between $S(X_1,Y_1)$ and $S(X_0,Y_1)$. 

Let us expand on this calculation  a little bit by looking at the covariance terms
\begin{align*}
&COV(R^1(n)-R^1(0),R^2(n)-R^2(0))=\\
&COV(\sum_{i=1}^nS(X_i,Y_i),\sum_{i=1}^nS(X_{i-1},Y_i))=\\
&\sum_{i=1}^n COV(S(X_i,Y_i),S(X_{i-1},Y_i))+
\sum_{i=1}^{n-1}COV(S(X_i,Y_i),S(X_i,Y_{i+1}))=\\
&n COV(S(X_1,Y_1),S(X_0,Y_1))+(n-1)COV(S(X_1,Y_1),S(X_1,Y_2))
\approx\\
&2n\cdot  COV(S(X_1,Y_1),S(X_0,Y_1)).
\end{align*}

We can  work out the same calculation for all pairs $i,j\leq k+1$ of random
walks $R^i$ and $R^j$ provided $i\neq j$. The covariance is always
about $2n\cdot COV(S(X_1,Y_1),S(X_0,Y_1))$.
Hence, the multi-dimensional Brownian motion in our approximation 
\eqref{approxB} has a covariance matrix with all non-diagonal
 entries equal
to $COV(S(X_1,Y_1),S(X_0,Y_1))$. Hence, we find
$$COV(B^i(t),B^j(t))=t\cdot 2COV(S(X_1,Y_1),S(X_0,Y_1))
$$
for $i\neq j$
and 
\[
VAR[B^i(t)]=t\cdot VAR[S(X_1,Y_1)].
\]
This is a simple covariance structure since there is one
common factor contained in all the Brownian motions
$B^i(t)$ if $COV(S(X_1,Y_1),S(X_0,Y_1))$ is non-zero.
In other words, we find that 
there exists a multidimensional Brownian motion
$\vec{W}(t)=(W^1(t),W^2(t),\ldots,W^{k+1}(t))$ with all components 
independent and having variance
$$VAR[W^i(t)]=t 
\left(VAR[S(X_1,Y_1)]-2COV(S(X_1,Y_1),S(X_0,Y_1))\right)
$$
so that 
$$\vec{B}(t)=\vec{W}(t)+V(t)I$$
where $I$ represents the $k+1$-dimensional
identity matrix and $V(t)$ is a one dimensional
Brownian motion which is independent of $\vec{W}(t)$.
(To understand more in detail how we get
the $V$ and the $W$ consider this: simply take a Brownian motion $V$
with $0$ expectation starting at $0$
 and independently a collection of Brownian
motions $W^1,W^2,\ldots,W^{k+1}$ which among themselves are i.i.d. 
Also, the $W_j$'s have all $0$ expectation and start at $0$.
You can now chose all these Brownian motion,
so that $V+W^1,V+W^2,\ldots,V+W^k$ have the same correlation
structure then the collection $B^1,B^2,\ldots,B^{k+1}$. Then,
because you deal with gaussian process you have the same distribution.
Then you can make a perfect coupling between the $V+W^1,V+W^2,\ldots,V+W^k$
and the $B^1,B^2,\ldots,B^{k+1}$ since they have the same distribution).
We also have 
\begin{equation}
\label{1054}
VAR[V(t)]=2tCOV(S(X_1,Y_1),S(X_1,Y_2)).
\end{equation}
Hence, we find that
\begin{equation}\label{target}
\begin{split}
L^k(\vec{B}):&=\max_{0<t_1<\ldots<t_k<1}\left((B^1(t_1)-B^1(0))+
(B^2(t_2)-B^2(t_1))+\ldots+(B^{k+1}(1)-B^{k+1}(t_k))\right)\\
&=V(1)-V(0)+L^k(\vec{W}),
\end{split}
\end{equation}
where
\begin{equation}
\label{I}
L^k(\vec{W}):=\max_{0<t_1<\ldots<t_k<1}\left((W^1(t_1)-W^1(0))+
(W^2(t_2)-W^2(t_1))+\ldots+(W^{k+1}(1)-W^{k+1}(t_k))\right).
\end{equation}
It is known (see \cite{MR1818248,MR1830441,MR1887169}), that  the rescaled $L^k(W)$ goes asymptotically
 to Tracy-Widom, more precisely, 
\begin{equation}
\label{II}
k^{1/6}\cdot (L^k(\vec{W})-2\sqrt{k})\rightarrow F_{TW}
\end{equation}
where convergence is in law as $k\rightarrow\infty$.
Hence, there is a sequence of random variable $\mathcal{E}_k$ which goes
to $0$ in probability as $k\rightarrow\infty$, and a sequence
of variables $F_{TW}^k$ each having same Tracy-Widom distribution,
so that
\begin{equation}
\label{losely}
L^k(\vec{W})=2\sqrt{k}+\frac{F_{TW}^k}{k^{1/6}}+\frac{\mathcal{E}_k}{k^{1/6}}.
\end{equation}
 We can now rewrite 
expression \eqref{gissv2} using \eqref{target} and get
\begin{equation}
\label{gissv3}
\left|\frac{L^k_n(S)-n\cdot E[S(X_1,Y_1)]}{\sqrt{n}}-(V(1)-V(0)+L^k(\vec{W}))
\right|\leq 
2\delta_n(k+1)+ \frac{2C\cdot k}{\sqrt{n}}
\end{equation}
which, with the help of equation \eqref{losely}, becomes
\begin{equation}
\label{gissv4}
\left|\frac{L^k_n(S)-n\cdot E[S(X_1,Y_1)]-2\sqrt{nk}}{\sqrt{n}}-(V(1)-V(0)+
\frac{F_{TW}^k}{k^{1/6}})
\right|\leq 
2\delta_n(k+1)+ \frac{2C\cdot k}{\sqrt{n}}+\frac{\mathcal{E}_k}{k^{1/6}}.
\end{equation}
For the above inequality to be useful, we need the right side of it to be smaller
than $\frac{1}{k^{1/6}}$.  This will depend on the value of $\delta_n$.
Now, we want to let $k$ go to infinity at a slower rate then $n$.
Most papers on closeness of random walk and Brownian motion are for fixed
dimension. Also, in our case, any $k$ consecutive components of the multi-dimensional random walk $\vec{R}_i$ are correlated.  However,  still assuming that the approximations hold for fixed dimension and correlated multidimensional random walk, we could take $\delta_n$ to be
\[
\delta_n=\frac{\log n}{n^{1/4}}
\]
whilst 
\[
\epsilon_n=\exp(-a \sqrt{\log(n)})
\]
for $a>0$ a constant not depending on $n$ or $k$.   In this case for the inequality \eqref{gissv4} to be useful,  we would need
\[
2\delta_n(k+1)+ \frac{2C\cdot k}{\sqrt{n}}+o\left(\frac{1}{k^{1/6}}\right)=
2\frac{\log(n)(k+1)}{n^{1/4}}+ \frac{2C\cdot k}{\sqrt{n}}+o\left(\frac{1}{k^{1/6}}\right)
\]
to be of smaller order than $\frac{1}{k^{1/6}}$.
This is the case when $k=n^\alpha$, with $\alpha$ small enough.

%
%$$L_n^k(S)= nE[S(X_1,Y_1)]+
%\sqrt{n}\sigma_{S(X_1,Y_1)}\left(
%\max_{(c_0,c_1,\ldots,c_k)} 
%\frac{(R^0(c_1)-R^0(1)-v(c_1)-v(1))}{\sqrt{n}\sigma_{S(X_1,Y_1)}}+
%\frac{R^1(c_2)-R^1(c_2)-(v(c_2)-v(c_1))}{\sqrt{n}\sigma_{S(X_1,Y_1)}}+
%\frac{R^1(c_3)-R^1(c_2)-(v(c_3)-v(c_2))}{\sqrt{n}\sigma_{S(X_1,Y_1)}}
%+\ldots +. 
%\frac{R^1(c_3)-R^1(c_2)-(v(c_3)-v(c_2))}{\sqrt{n}\sigma_{S(X_1,Y_1)}} \right) \%approx
%nE[S(X_1,Y_1]+\sqrt{n}\sigma_{S(X_1,Y_1)}
%\left(\max_{0<t_1<t_2<\ldots<t_k\leq 1}
%(B^0(t_1)-B^0(0))+(B^1(t_2)-B^2(t_1))+\ldots+(B^k(1)-B^k(t_k)) \right)
%$$

\section{Letting the number of gaps $k=n^{\alpha}$ and normal letters}\label{s:4}

In this section, we are going to treat the case where the number of gaps
$k$ increases to infinity with $n$ and the simplest such scenario is the case of  
\[
k=n^\alpha
\] 
for some small $\alpha>0$ which we assume is not depending on $n$. Again,  we look
for the best alignment score under all alignments with the prescribed number of gaps being $k$ and the gaps appear only in the sequence of $X$'s, thus we allow gaps to align with $Y$'s. 

The highest alignment score among all alignments which align exactly $k$ letters of $Y=Y_1\ldots Y_n$ with gaps, but no letter of $X=X_1\ldots X_n$ with gaps, is again denoted by
\[
L^k_n(S)=\text{Optimal alignment score with exactly } k \text{ gaps in } X.
\]
When the number of gaps is fixed, in determining which alignment is optimal, the penalty for gaps does not matter much.  Indeed, the gap penalty will affect the optimal alignment only by a term of order $n^{\alpha}$.
Moreover,  if we assume that the penalty of aligning a letter with a gap does not depend on the letter, we can simply assume that the gap penalty is actually equal to $0$.

Therefore, the only part of the scoring functions which matters is the $2\times 2$ symmetric matrix
\[
\left(\begin{array}{cc}
S(a,a)&S(a,b)\\
S(b,a)&S(b,b)
\end{array}
\right).
\]
The space of all such  $2\times 2$ symmetric matrices constitute a $3$ dimensional space where a basis is formed by the matrices
\[
S_{0}=\left(\begin{array}{cc}
1&1\\
1&1
\end{array}
\right),\,
S_{1}=\left(\begin{array}{cc}
1&0\\
0&-1
\end{array}
\right),\,
S_{2}=\left(\begin{array}{cc}
1&-1\\
-1&1
\end{array}
\right)
\]
In addition, if we define two scoring matrices to be \emph{equivalent} if the optimal alignments are the same, then it is easy to see that for any arbitrary scoring function $S$, the scoring itself $S$ and  $S-aS_{0}-bS_{1}$ are equivalent.  Thus the main task is to understand what happens under alignments with the scoring function $S_{2}$.  Thus we assume for the rest of the section that 
the scoring function $S$ is equal to $S_2$, so that:
\[
\left(\begin{array}{cc}
S(a,a)&S(a,b)\\
S(b,a)&S(b,b)
\end{array}
\right)=\left(\begin{array}{cc}
1&-1\\
-1&1
\end{array}
\right).
\]
Now, in the case considered here of a two dimensional alphabet, that scoring function can be viewed
as a product scoring function by observing that if we introduce a numeric code for the letters, namely take $a$
to be $+1$ and $b$ to be $-1$. Then we find
that
\[
 S(X_i,Y_j)=X_i\cdot Y_j.
 \]

Using this in \eqref{target}, the key observation is that the part $V(t)$ has covariance given in
formula \eqref{1054} by
\[
VAR[V(t)]=2COV(S(X_1,Y_1),S(X_1,Y_2)),
\]
which in the present case becomes
\begin{align*}
&VAR[V(t)]=COV[X_1Y_1,X_1Y_2]=E[X_1^2Y_1Y_2]-E[X_1Y_1]\cdot E[X_1Y_2]=\\
&=E[X_1^2]E[Y_1]E[Y_2]-E[X_1]E[Y_1]\cdot E[X_1]E[Y_2]=0
\end{align*}
due to independence and mean $0$ of $X_1,Y_1,Y_2$. Hence, the term $V(t)$ disappears from
inequality \eqref{gissv4} for this choice of the scoring function. 

Next, we notice that for a one dimensional ``uncorrelated'' random walk,  the best coupling of the walk with a Brownian motion in \eqref{approxB} can be attained with $\delta_n=\frac{\ln(n)}{n^{1/4}}$ (see  \cite{MR0402883}).

With this choice of $\delta_n$, the term $\frac{2C\times k}{\sqrt{n}}$ in \eqref{gissv4} is of order $O(n^{\alpha-1/2})$, thus is much smaller than $2\delta_n(k+1)$ which is of order $O(n^{\alpha-1/4}\ln(n))$.
However, since we have to correlate a large number of random walks (exactly $k$ and not independent), we expect $\delta_n$ to be larger than $O(\ln(n)/n^{1/4})$.   

A different look at the dominating term on the right hand side of \eqref{gissv4}, gives that 
\begin{equation}
2\delta_n\cdot (k+1)= O(n^{\alpha-1/4}\ln(n))=o(1/k^{1/6}) 
\end{equation}
which is enough (at least heuristically) to guarantee that the rescaled optimal score is asymptotically Tracy-Widom.  

In order to deal with the approximation of the random walk by Brownian motions, we need to get good couplings of the random walk with Brownian motions.   This is what we want to do next.  

We should also quickly mention, that at this stage, it should be heuristically clear why we should expect 
Tracy-Widom as limiting distribution as soon as the constant $\alpha>0$
is small enough. Indeed, when $\alpha>0$ is very small, then $\delta_n$ should 
be as ``close as we want'' to the formula for the one dimensional 
random walk with independent increments.  Now, for the one-dimensional  random walk, 
it is known to allow a coupling to a Brownian motion with
 $\delta_n=\frac{\ln(n)}{n^{1/4}}$.  Again, the only problem is  to find
a good coupling of the multidimensional random
walk $\vec{R}(t)$ to a multidimensional Brownian motion.
The random walk $\vec{R}(t)$ has $k+1$ components. Furthermore,
there is dependency among the increments of the random walk.
Each step (increment) of the random walk, can depend on
$k$ steps  before and after.  Hence, $\vec{R}(t)$ is independent
of $\vec{R}(t+k+1),\vec{R}(t+k+2),\vec{R}(t+k+3),\ldots$
for all $t\geq 0$.   

The main question is thus how well can a random walk with dependencies be approximated by Brownian motion?   (See  \cite{MR0375412,MR0402883} for the one dimensional case, \cite{MR996984} for severeal dimensions and  \cite{MR2068476,MR666546} for overviews.) 

Is it possible to find such a coupling when the dimension of the random walk grows to infinity? This is delicate as the constants involved depend on the dimension and thus are not very useful in our case.  

Unfortunately we can not treat both problems simultaneously and it seems highly technical to deduce such a result in great generality.    Instead, we present here a simple approach for  a simplified situation.

First, we use the fact  that our scoring function is a product. Furthermore, \emph{for the rest of the paper we also  assume the variables $Y_1,Y_2,\ldots$ to be  i.i.d. normal with 
\[
Y_i\sim \mathcal{N}(0,1)
\]
instead of binary variables. The variables $X_1,X_2,\ldots$ 
are still binary variables in this section with $P(X_i=-1)=P(X_i=+1)=0.5.$} 

Again, since  $E[X_i]=0$ the component $V(1)-V(0)$  in the equation \eqref{gissv4} dissapears. 

The normal assumption on $Y$'s is in place here to simplify the calculations. We believe
that macroscopically it amounts to the same as taking these variables binary.   The fact that the variables $Y_1,Y_2,\ldots$ are standard normal instead of binary has a very important consequence on the calculation side, namely, the optimal alignment still tries to align
the maximum number of letter with the same sign and also will try to get that those $Y_i$'s 
which get aligned with an $X_i$ of opposite sign to be small in absolute value. 

The key strategy of the proof is to condition on the sequence $X$ and exploit the normal structure of the $Y$.  To set up the ground, let $X=(X_1,X_2,\ldots,X_n)$ and let $x=(x_1,x_2,\ldots,x_n)\in\{-1,1\}^n$
be a non-random binary vector.   We are going to work conditional on 
$X=x$.    

Now, consider our random walk $\vec{R}$ in the current case where  $S(a,b)=a\cdot b$. We have
\[
\vec{R}(i)=(R^1(i),R^2(i),\ldots,R^{k+1}(i))^{t},
\]
where
$\vec{R}(0)=(0,0,0,\ldots,0)^{t}$
and
\[
\Delta R(1)=\vec{R}_1=\left(
\begin{array}{l}
X_1\cdot Y_1\\
X_0\cdot Y_1\\
X_{-1}\cdot Y_1\\
\ldots\\
X_{1-k}\cdot Y_1
\end{array}\right)
,
\Delta R(2)=\vec{R}_2-\vec{R}_1=\left(
\begin{array}{l}
X_2\cdot Y_2\\
X_1\cdot Y_2\\
X_{0}\cdot Y_2\\
\ldots\\
X_{2-k}\cdot Y_2
\end{array}\right),
\dots
%\Delta R_3=\vec{R}_3-\vec{R}_2=\left(
%\begin{array}{l}
%X_3\cdot Y_3\\
%X_2\cdot Y_3\\
%X_{1}\cdot Y_3\\
%\ldots\\
%X_{3-k}\cdot Y_3
%\end{array}\right)
%,\ldots
\]
Looking at the above expression for the random walk $\vec{R}$,
we see that the steps are dependent but once we condition the $X_i$'s to be constants, i.e. $X=x$ and  $Y_1,Y_2,\ldots$ to be i.i.d $N(0,1)$, the $\Delta R(i)$'s become independent normal vectors. Hence, 
conditional on $X=x$,  we get that the multidimensional random walk $\vec{R}$ has  independent normal increment with expectation $0$.   

This random walk is not yet a Brownian motion restricted to discrete times but we are going to show that we can construct such a Brownian motion which is a good approximation.  
In order to do this \emph{we consider at first the random walk $\vec{R}(t)$ restricted to the times $t$ which are a multiple of 
\[
j:=[n^{\beta}]=\text{time-mesh size for coupling the random walk}
\]}
with some small $\beta>0$ which will be determined later and $[x]$ denoting the integer part of $x$.  

Thus, we will consider 
\[
\vec{R}(j),\vec{R}(2j),\vec{R}(3j),\ldots, R(n),
\]
where to simplify notation we assume $n$ to be a multiple of $j$, though we can formally replace $n$ by $jn_{j}$ with $n_{j}=[n/j]$ and run essentially the same argument.   

Using the independence of  the increments conditional on $X=x$,  we calculate the covariance
matrix of the position of the random walk at time $j$ in the following way:
\begin{align}
\label{schneesturm}&COV[\vec{R}_{j}|X=x]=
COV\left[\sum_{i=1}^j\Delta R_i|X=x\right]=
\sum_{i=1}^jCOV[\Delta R_i|X=x]\\
&=\sum_{i=1}^j
\left(\begin{array}{c}
x_{i}\\
x_{i-1}\\
x_{i-2}\\
...\\
x_{i-k+1}
\end{array}
\right)\cdot (x_{i},x_{i-1},x_{i-2},\ldots,x_{i-k+1}).
\end{align}
The key fact here is that with very high probability the last sum in the above equation is approximately $jI$ where $I$ is the $k+1$ identity matrix.   

Now, using this, and the fact that $jI$ is the covariance matrix of a standard $k+1$-dimensional Brownian motion
at time $j$ we will show that there is a coupling between the random walk and a standard Brownian motion such that  $\vec{R}(j)$ is close to the Brownian motion position at time  $j$.   Moreover the argument can be extended to include also the increments of the random walk.  More precisely, we show that the coupling can be chosen such that the increments
\[
\vec{R}(2j)-\vec{R}(j),\vec{R}(3j)-\vec{R}(2j),
\vec{R}(4j)-\vec{R}(3j),\ldots,\vec{R}(n)-\vec{R}(n-j)
\]
are also close to the increment of the coupled Brownian motion.  

To construct such a coupling, we will build a Brownian motion
$$\vec{C}(t)=(C^1(t),C^2(t),\ldots,C^{k+1}(t))^{t}$$ by coupling at the times $j,2j,3j,\ldots,n$ at first. 
More precisely, we couple $\vec{R}(j)$ with $\vec{C}(j)$, then the increment $\vec{R}(2j)-\vec{R}(j)$
with the increment $\vec{C}(2j)-\vec{C}(j)$ and then we continue coupling the next increment $\vec{R}(3j)-\vec{R}(2j)$ with $\vec{C}(3j)-\vec{C}(2j)$ and so on and each step we do this independently of the construction at the previous steps.  Recall that we do this conditionally on $X=x$.    

The key is that the construction at each step is a consequence of Lemma~\ref{coupling} in the next Section \ref{howto} where we show how to couple a normal random variable having covariance close to a multiple of the identity. Essentially the main assertion here is that if $\vec{A}$ is a normal vector such that 
\[
|Cov(A)-\sigma^{2}I_{d}|\le \epsilon
\]
then we can find a normal vector $\vec{B}$ of i.i.d normal components $N(0,\sigma^{2})$ such that $|Cov(\vec{A}-\vec{B})|\le \epsilon$.  Here we use the operator norm $|A|$ for a matrix and measure the closeness in the operator norm, in particular all components of $\vec{A}-\vec{B}$ have variance at most $\epsilon$.   

%That normal vector becomes then the increment of the standard Brownian motion,
% denoted by $\vec{C}(2j)-\vec{C}(j)$. This coupling is independent
%of the previous increments and depends only on $\vec{R}(2)-\vec{R}(j)$,
%at least  conditional on $X=x$.) Similarly, for every 
%$i=1,2,3,\ldots, \frac{n}{j}$, we couple $\vec{R}(ij+j)-\vec{R}(ij)$
%with a $k+1$-dimensional normal vector having i.i.d. entries, expectation $0$
%and variances $j$. This normal vector is  the increment
%$\vec{C}_{ij+j}-\vec{C}_{ij}$. How to obtain this coupling is
% described in Subsection \ref{howto}. 
%There, it is shown how to couple  a normal vector having covariance matrix close to
%$Ij$ to a normal vector having $Ij$ as covariance.

%In particular, if the covariance matrix is closer to
%$Ij$ than $\epsilon>0$, then a coupling 
%can be found so that each coordinate of the difference
% has variance less or equal
%to $\epsilon$. (See Lemma \ref{coupling}).\\
Once all the increments $\vec{C}((i+1)j)-\vec{C}(ij)$ are defined, we define 
$\vec{C}(t)$ for $t=j,2j,3j,\ldots,n$ as the sum
\[
\vec{C}(t):=\sum_{i=1}^{t/j}(\vec{C}((i+1)j)-\vec{C}(ij)).
\]
Conditional on $X=x$, the pieces
\begin{equation}
\label{haha}
\vec{R}((i+1)j)-\vec{R}(ij)
\end{equation}
for $i=1,2,\ldots, n/j$ are independent. Since
\[
\vec{C}((i+1)j)-\vec{C}(ij)
\]
depend only on \eqref{haha}, this implies that conditional on $X=x$, 
the sequence of random vectors
\[
\vec{C}(0),\vec{C}(j),\vec{C}(2j),\vec{C}(3j),\ldots,\vec{C}(n)
\]
constitutes a multidimensional standard Brownian motion restricted to the times $0,j,2j,3j,\ldots,n$. We can for instance use the construction of Levy for the Brownian motion to extend these vectors to a Brownian motion such that the increments are given exactly by $\vec{C}((i+1)j)-\vec{C}(ij)$ for $i=1,2,\dots, n/j$.  We do not look for other details about the Brownian motion during the times $[ij,(i+1)j]$, we will simply use some basic estimates for the maximum of Brownian motions between these time intervals.   For instance using the reflection principle we can show that 
\[
P(\max_{u\in [s,t]}|\vec{C}(u)-\vec{C}(s)|\ge \lambda)\le const(k+1)e^{-\lambda^{2}/(2(t-s))},\quad \forall \lambda>0. 
\]
In particular what this yields is that for some constant $c$ independent of $n$ and $\alpha$
\[
P(\max_{u\in [ji,j(i+1)]}|\vec{C}(u)-\vec{C}(s)|\ge \log(n)\sqrt{j})\le cn^{\alpha}e^{-\log^{2}(n)/2}
\]
thus the deviation of the Brownian motion on each of these intervals is very well controlled.

%For each time interval $[ij,(i+1)j]\subset [0,n]$, we 
%then simulate the standard Brownian motion during that time
%to only depend on $\vec{C}(ij+j)-\vec{C}(ij)$.  That is we add a Brownian 
%bridge from time $ij$ to time $(i+1)j$. We do this
% for every $i=1,2,ldots,\frac{n}{j}$.
% We do not take care of coupling 
%this Brownian bridge closely to the random walk $\vec{R}_t$
%during that time interval. Anyhow, for
%those time-intervals of length $j$, a Brownian motion is unliquely to go
%further than $\log(n)\sqrt{j}$, so this will not bother us much.\\

Now define a few events related to our coupling which will take care of the main estimates.   

\begin{itemize}
\item Let $\bf{F^n}$ be the event that none of the components
of the standard multidimensional Brownian motions $\vec{C}(t)$ goes further than $\log(n)\sqrt{j}$ during any
of the time intervals $[ij,(i+1)j]\subset [0,n]$ for $i=0,1,2,\ldots,n/j$
That is let $F^n_l$ be the event that
\[
\sup_{s\in[0,j]}|C^l(ij+s)-C^l(ij)|\leq \log(n)\sqrt{j}
\]
for all $i=1,2,\ldots,n/j$.  Set 
\[
F^n=\cap_{l=1}^{k+1} F^n_l.
\]
\item Let $\bf{G^n}$ be the event that the multidimensional random walk 
$\vec{R}(t)$ does not move further than $\log(n)\sqrt{j}$ during any
of the time intervals $[ij,(i+1)j]\subset [0,n]$.
Hence, let $G^n_l$ be the event
\[
\sup_{s\in[0,j]}|R^l(ij+s)-R^l(ij)|\leq \log(n)\sqrt{j}
\]
for all $i=1,2,\ldots,n/j$.
Let 
$$G^n=\cap_{l=1}^{k+1} G^n_l.$$
\item{}The event $\bf{A^n}$ takes care  of the fact that the covariance matrix
of $\vec{R}(ij)-\vec{R}((i-1)j)$ for all $i=1,2,\ldots,n/j$ is close
to $jI$,were $I$ denotes the $(k+1)\times(k+1)$ identity matrix. More specifically,
let thus $A^n_1$  be the event that
\[
|COV[\vec{R}(j)]-jI|\leq C\cdot k\sqrt{j},
\]
where $C>0$ is a constant which will be defined later.
Similarly for any integer $i$ so that 
\begin{equation}
\label{ij}
[ij,(i+1)j]\subset[0,n]
\end{equation}
holds,
let $A^{n}_i$ be the event that
\[
|COV[\vec{R}((i+1)j)-\vec{R}(ij)]-
jI|\leq C\cdot k\sqrt{j}.
\]
Finally let $A^n$ be the event 
\[
A^n=\cap_{i=1}^{n/j} A^n_i.
\]
\item The next event $\bf{D^n}$ assures that the random walk $\vec{R}(t)$ and 
the Brownian motion $\vec{C}(t)$ are close to each 
other at times which are multiples of $j$. More specifically,
let $D^n_l$ be the event that the $l$-th component
of $\vec{C}(t)-\vec{R}(t)$ at times which are multiples of $j$
is bounded as follows:
\[
\max_{i=1,\ldots,n/j}|C^{l}(ij)-R^l(ij)|\leq 
\log(n) \cdot\sqrt{k}\frac{\sqrt{n}}{j^{1/4}}.
\]

Let $D^n$ be intersection of these for  all $l$, precisely:
$$D^n=\cap_{l=1}^{k+1} D^n_l.$$
\end{itemize}

Now, we summarize the picture we have obtained so far.  
Rescale the random walk and the Brownian motion $\vec{C}_t$
by a factor $\sqrt{n}$ and the time by a factor $n$ by setting for $t\in[0,1]$
\begin{equation}
\label{greek}\vec{B}(t):=\frac{\vec{C}(tn)}{\sqrt{n}}.
\end{equation}
By scaling the Brownian motion, we know that $\vec{B}(t)$ is again a Brownian motion.  When the events  $D^n$, $F^n$ and $G^n$ all hold,  inequality \eqref{approxB}
is satisfied for the multi-dimensional 
Brownian motion $\vec{B}_t:[0,1]\rightarrow\mathbb{R}^k$ with
$\delta_n$ given as
\[
\delta_n:=4\max\left\{\frac{\log(n)\sqrt{j}}{\sqrt{n}},
\frac{\log(n)\sqrt{k}}{j^{1/4}}\right\}
\]
where again, the number of gaps $k$ and the time step are given by 
\[
k=n^\alpha \text{ and } j=n^\beta.
\]
We can now go back to inequality \ref{gissv4},
to find that:
\begin{align}\label{turkish}
&\left|\frac{L^k_n(S)-n\cdot E[S(X_1,Y_1)]-2\sqrt{nk}}{\sqrt{n}}-
\frac{F_{TW}^k}{k^{1/6}}
\right|\leq \\
&\leq 9 
\max\{\frac{k\log(n)\sqrt{j}}{\sqrt{n}},
\frac{k\log(n)\sqrt{k}}{j^{1/4}}\}+ +\frac{\mathcal{E}_k}{k^{1/6}}
\end{align}
holds with probability at least
$1-\epsilon_n$
where
\begin{equation}
\label{epsilon_n}
\epsilon_n:=P((B^{n})^{c})+P((F^{n})^{c})+P((G^{n})^{c})\leq 
P((A^{n})^{c})+P((D^{n})^{c}|A^{n})+P((F^{n})^{c})+P((G^{n})^{c}).
\end{equation}
Recall also, that $\mathcal{E}_1,\mathcal{E}_2,\ldots$ designates a sequence
of random variables which go to $0$ in probability. The sequence
$F_{TW}^1,F^{2}_{TW},\ldots$ designates on the other hand a sequence
of random variables which are all identically distributed according
to Tracy-Widom.\\
We eliminated the term $\frac{2C\cdot k}{\sqrt{n}}$ by increasing the coefficient
in front of $\max\left\{\frac{k\log(n)\sqrt{j}}{\sqrt{n}},
\frac{k\log(n)\sqrt{k}}{j^{1/4}}\right\}$. 

Furthermore, from inequality \ref{turkish}, we obtain the important conclusion that 
in order to show that the re-scaled optimal alignment score 
with $k=n^{\alpha}$ imposed gaps goes to Tracy Widom
we only need to prove prove two things:

{\bf I:} $\epsilon_n$ goes to zero, hence
\[
P((A^{n})^{c})+P((D^{n})^{c}|A^{n})+P((F^{n})^{c})+P((G^{n})^{c})\rightarrow 0
\]
as $n\rightarrow \infty$, which we solve in Section \ref{highprob}.

{\bf II:} The term 
\begin{equation}
\label{reference}
\max\left\{\frac{k\log(n)\sqrt{j}}{\sqrt{n}},
\frac{k\log(n)\sqrt{k}}{j^{1/4}}\right\}
\end{equation}
in the bound in \ref{turkish} is of order less
than $\frac{1}{k^{1/6}}$.  This is the place where we choose $\alpha$ and $\beta$.  

We have to verify that there exists a choice
$\alpha>0$ and $\beta>0$ making expression \ref{reference} 
less than order $1/k^{1/6}$.  
In other words, we want  
\begin{equation}
\label{O}
O\left(\max\left\{\frac{n^\alpha\log(n)\sqrt{n^\beta}}{\sqrt{n}},
\frac{n^\alpha\log(n)\sqrt{n^{\alpha}}}{n^{1/4}}\right\}\right)<
O\left(\frac{1}{n^{\alpha/6}}\right)
\end{equation}
 which is satisfied for $n$ large enough as soon as
\[ 
\frac{20}{3}\alpha< \beta<1-\frac{7}{3}\alpha.
\]
This is of course possible if and only if 
\[
\frac{20}{3}\alpha< 1-\frac{7}{3}\alpha
\]
which leads to $\alpha< \frac{1}{9}$.  Thus, for any $\alpha<\frac{1}{9}$,
we  have convergence of the re-scaled optimal alignment score 
to Tracy-Widom.   Hence, we are able to prove our result
provided there are less than order $n^{1/9}$ gaps in the optimal alignment!
Let us formulate this fact in the form of a Theorem.  

\begin{theorem}\label{t:1}
Let $\alpha\in[0,1/9)$ be a constant. Let $X_1,X_2,\ldots$ be an i.i.d.
sequence of binary variables so that
$$P(X_i=1)=P(X_i=-1)=1/2$$
and let $Y_1,Y_2,\ldots$ be an i.i.d.
 sequence of standard normal random variables.
Let $S_2$ denote the scoring function defined by $S_2(c,d)=c\cdot d$
for any two letters $c,d$ of our binary-alphabet.
Let $k=[n^\alpha]$. Then, the optimal alignment score with constrain
to have $k$ gaps and properly rescaled as 
\begin{equation}
\label{expression}
\frac{L^k_n(S)-n\cdot E[S(X_1,Y_1)]-2n^{1/2+\alpha/12}}{n^{1/2-\alpha/6}}
\end{equation}
converges in law to Tracy-Widom.
\end{theorem}

\begin{proof}
Multiplying \eqref{turkish} by $k^{1/6}$, we find that
that
\begin{align}\label{turkish2}
&\left|k^{1/6}\cdot\frac{L^k_n(S)-n\cdot E[S(X_1,Y_1)]-2\sqrt{nk}}{\sqrt{n}}-
F_{TW}^k
\right|\leq \\
&\leq k^{1/6}\cdot 9 
\max\{\frac{k\log(n)\sqrt{j}}{\sqrt{n}},
\frac{k\log(n)\sqrt{k}}{j^{1/4}}\}+ \mathcal{E}_k
\end{align}
holds with probability at least
$1-\epsilon_n$, In the next section it is proven that $\epsilon_n$ goes to 
$0$ as $n\rightarrow\infty$. (Recall also inequality \eqref{epsilon_n}).
Furthermore, since $\alpha<1/9$,
we have from \eqref{O} that
$$k^{1/6}\cdot 9 
\max\{\frac{k\log(n)\sqrt{j}}{\sqrt{n}},
\frac{k\log(n)\sqrt{k}}{j^{1/4}}\}
$$ goes to $0$ as $n\rightarrow\infty$. Hence, since also
$\mathcal{E}_k$
go to $0$ in probability, we have that
$$\left|k^{1/6}\cdot\frac{L^k_n(S)-n\cdot E[S(X_1,Y_1)]-2\sqrt{nk}}{\sqrt{n}}-
F_{TW}^k
\right|$$
goes to $0$ in probability, when $n\rightarrow \infty$. Hence,
since $F_{TW}^k$ always has Tracy-Widom distribution,
it follows, that
\begin{equation}
\label{k1/6}k^{1/6}\cdot \frac{L^k_n(S)-n\cdot E[S(X_1,Y_1)]-2\sqrt{nk}}{\sqrt{n}}
\end{equation}
converges in Law to Tracy-Widom as $n\rightarrow\infty$.
But, when we replace $k$ by $n^\alpha$ in  expression
\ref{k1/6}, we obtain  expression \eqref{expression} from our theorem.
So, we just finished proving that \eqref{expression}
converges to Tracy-Widom in Law.
\end{proof}

{\bf Remark.} We believe that taking the variables $Y_1,Y_2,\ldots$
binary instead of standard normal our
theorem should holds for  $\alpha>0$ small enough. However, the proof for approximating
the multidimensional random walk $\vec{R}^j$ by a multidimensional
Brownian motion will be much more difficult.  The intricate part is very similar to passing from 
the CLT to Berry-Essen type CLT theorem is when
there are non-finite correlations or dimension goes
to infinity at the same time as $n$ (see \cite{bentkus}, \cite{goetze}).

Now we got Tracy-Widom limiting distribution for the binary case,
by using the scoring function $S_2$ and fixing the number
of gaps. If instead we use a scoring function 
\[
S=S_2+a_0S_1+a_1S_1
\]
then we get an additional normal term on the scale $\sqrt{n}$.
Indeed, $S_0$ and $S_1$ do not affect which alignment is optimal, thus,
\[
L_n^k(S_2+a_0S_0+a_1S_1)=
L_n^k(S_2)+L_n^k(a_0S_0+a_1S_1)
\]
As argued before, $L_n^k(a_0S_0+a_1S_1)$
depends only on the number of $a$'s and $b$'s in each
of the sequences $X$ and $Y$. It is even a linear functional
of these numbers of occurrences and hence asymptotically
normal on the scale $\sqrt{n}$.  On, the other hand, according to our theorem,
we have that after rescaling $L_n^k(S_2)$ is asymptotically Tracy-Widom 
on the scale  $n^{1/2-\alpha/6}$, (provided $\alpha\in[0,1/6)$ and
$k=n^\alpha$). Thus, there is a larger normal part and Tracy-Widom 
on a lower scale. But, for optimal alignment of DNA sequences,
one should always set $a_0=a_1=0$, since the part $a_0S_0+a_1S_1$ does not influence
which alignment is optimal. In  that case when $a_0=a_1=0$, 
the limiting distribution is purely Tracy-Widom.

\emph{The main question to ask here is what happens if we allow the number of gaps to be of order $O(n)$ which we leave as an open problem here. }

\section{Technical results}\label{s:5}

\subsection{ How to couple normal random vectors with close
covariance matrix}\label{howto}

The main coupling result is contained in the following Lemma.  

\begin{lemma}
\label{coupling}
Let $\vec{v}$ be a $d$-dimensional normal random vector with 
expectation $0$. Let $\epsilon>0$ and $j>0$.  Assume
that  
\begin{equation}\label{e:cov}
|COV[\vec{v}]-jI|\leq \epsilon.
\end{equation}
Then there exists a $d$-dimensional normal random vector $\vec{N}$ with covariance matrix $jI$ and expectation $0$ so that
\[
|COV[\vec{v}-\vec{N}]|\le\epsilon.  
\]
In particular all the components of $\vec{v}-\vec{N}$ have variance at most $\epsilon$.  
\end{lemma}

\begin{proof}

The main idea is to construct a normal vector simply by rescaling the vector $\vec{v}$.  The construction uses the fact that if $A$ is a linear transformation from  $\R^{d}$ into itself, then $A\vec{v}$ is again a normal random variable of covariance matrix $A\,COV(\vec{v})A'$.  In this framework we take $\Sigma=COV(\vec{v})$ and  
\[
\vec{N}=j^{1/2}\Sigma^{-1/2}\vec{v}.
\] 
Clearly, we now have
\[
COV(\vec{N})=j\Sigma^{-1/2}\Sigma\Sigma^{-1/2}=jI.
\]
Furthermore, 
\[
COV[\vec{v}-\vec{N}]=COV[(I-j^{1/2}\Sigma^{-1/2})\vec{v}]=(I-j^{1/2}\Sigma^{-1/2})\Sigma(I-j^{1/2}\Sigma^{-1/2})=(\Sigma^{1/2}-j^{1/2})^{2}.
\]
Taking the eigenvalues of $COV(\vec{v})$ to be $\lambda_{1},\lambda_{2},\dots,\lambda_{d}$, the condition \eqref{e:cov} reads as 
\[
\max|\lambda_{i}-j|\le \epsilon. 
\]
On the other hand since $\lambda_{i}\ge0$, we also know have that $|\sqrt{\lambda_{i}}-\sqrt{j}|\le |\sqrt{\lambda_{i}}+\sqrt{j}|$.  Thus the conclusion is reached from the fact that the eigenvalues of $\Sigma^{1/2}$ are  $\lambda_{i}^{1/2}$ and the following chain of inequalities
\[
|\sqrt{\lambda_{i}}-\sqrt{j}|^{2}\le |\sqrt{\lambda_{i}}-\sqrt{j}||\sqrt{\lambda_{i}}+\sqrt{j}|=|\lambda_{i}-j|\le \epsilon.  
\]

\end{proof}

\subsection{Proof that the events $A^n$, $D^n$, $F^n$ and $G^n$ all
have high probability}\label{highprob}

\begin{lemma}
\label{Aj}
We have that $A^n$ has probability close to one.
More precisely, assuming $\beta>\alpha$, then for any 
constant $\gamma<\beta-\alpha$, the following holds
\[
P((A^{n})^{c})=o(\exp(-n^\gamma)).
\]
\end{lemma}

\begin{proof}
Consider the sum on the right side of \eqref{schneesturm}.  Instead of taking the full sum, we are going to split the sum in the following form 
\[
\begin{split}
&\sum_{i=1}^{j}\left(\begin{array}{c}
x_{i}\\
x_{i-1}\\
x_{i-2}\\
\ldots \\
x_{i-k+1}
\end{array}
\right)\cdot (x_{i},x_{i-1},x_{i-k+1},\ldots,x_{i(k+1)-k})\\ 
&=\sum_{l=1}^{k}\sum_{i=1}^{j/k}\left(\begin{array}{c}
x_{i(k+1)+l}\\
x_{i(k+1)+l-1}\\
x_{i(k+1)+l-2}\\
\ldots \\ 
x_{i(k+1) +l-k+1}
\end{array}
\right)\cdot (x_{i(k+1)+l},x_{i(k+1)+l-1},x_{i(k+1)+ l-2},\ldots,x_{i(k+1)+ l-k+1})
\end{split}
\]
where for simplicity we assume that $j$ is a lot larger than $k$ and that $j$ is a multiple of $k$.

Since for each $1\le l\le k+1$, the sequence 
\[
\vec{X}_{i,l}=(X_{i(k+1)+l},X_{i(k+1)+l-1},\ldots,X_{i(k+1)+l-k+1}),
\]
are iid when $i$ runs from $1$ to $j/k$.   The covariance matrix of these vectors is the $k+1$-dimensional identity matrix $I$.  So,
\begin{equation}
\label{sumo}\sum_{i=1}^{j/(k+1)}\frac{\vec{X}_{i,l}\otimes \vec{X}_{i,l}}{j/(k+1)}
\end{equation}
is an estimate of the covariance matrix of $\vec{X}_{1,1}$. Hence,
expression \ref{sumo} is approximately equal to $I$ and
\begin{equation}
\label{thinnedsum}\sum_{i=1}^{j/k}\left(\begin{array}{c}
x_{i(k+1)+l}\\
x_{i(k+1)+l-1}\\
x_{i(k+1)+l-2}\\
\ldots \\ 
x_{i(k+1) +l-k+1}
\end{array}
\right)\cdot (x_{i(k+1)+l},x_{i(k+1)+l-1},x_{i(k+1)+ l-2},\ldots,x_{i(k+1)+ l-k+1})\approx \frac{j}{k+1} I
\end{equation}

Now, according to \cite{Vershinin2011}
the precision of such an approximation in operator norm
is the square root of the dimension times  the square root of the sample
size. More precisely,  we have
\[
\begin{split}
&P\left(
\left|\sum_{i=1}^{j/k}\left(\begin{array}{c}
x_{i(k+1)+l}\\
x_{i(k+1)+l-1}\\
x_{i(k+1)+l-2}\\
\ldots \\ 
x_{i(k+1) +l-k+1}
\end{array}
\right)\cdot (x_{i(k+1)+l},x_{i(k+1)+l-1},x_{i(k+1)+ l-2},\ldots,x_{i(k+1)+ l-k+1})-\frac{j}{k+1}I\right|\geq C\sqrt{k}\sqrt{\frac{j}{k+1}}
\right) \\
&\leq e^{-j/(k+1)}.
\end{split}
\]
Here, $C$ is  a constant which does not depend on $j$.
Hence, for the full sum, this means 
\[
P\left(|COV[\vec{R}_j|X=x]-jI|\geq Ck\sqrt{j}\right)
\leq (k+1)e^{-j/k+1},$$
So, we have proven that
$$P((A^{n}_{l})^{c})\leq (k+1)\exp^{-j/k+1}.
\]
Hence, since $P((A^{n}_i)^{c})=P((A^{n}_1)^{c})$ and since
$$A^n=\cup_{i=1}^{n/j}A^n_i,$$
we have
\begin{equation}
\label{merde}
P((A^{n})^{c})\leq \sum_{i=1}^{n/j}P((A^{n}_i)^{c})\leq \sum_{i=1}^{n/j}(k+1)e^{-j/(k+1)}=
\frac{n(k+1)}{j}e^{-j/(k+1)}
.\end{equation}
Since, $j=n^\beta$ and $k=n^\alpha$ with $\alpha>\beta$ it follows
that the dominant term in the bound on the right side of \ref{merde}
is $e^{-j/(k+1)}$ which is an exponential decaying function in  a fractional
power of $n$.

\end{proof}

The next lemma shows that when the event $A^n$ holds,
then $D^n$ has high probability
\begin{lemma}
We have that 
\[
P((D^{n})^{c}|A^n)\leq e^{-\log^{2}(n)/2}.
\]
In particular, $P((D^{n})^{c}|A^n)\xrightarrow[n\to\infty]{}0$.  
\end{lemma}

\begin{proof}
We work conditional on $X=x$.
We have that when the event $A^n$ holds,
then
\[
|COV(\vec{R}((i+1)j)-\vec{R}(ij))-jI|\leq C k\sqrt{j}
\]
where $I$ is a $k+1$ dimensional identity matrix.
So, our coupling of $\vec{R}((i+1)j)-\vec{R}(ij)$
to $\vec{C}((i+1)j)-\vec{C}(ij))$ is such that
the covariance matrix of the difference has norm at most
$Ck \sqrt{j}$. Thus, conditional on $X=x$,
and assuming $A^n$ holds, we get
\begin{equation}
\label{differences}
|COV[(\vec{R}((i+1)j)-\vec{R}(ij))- (\vec{C}((i+1)j)-\vec{C}(ij)))]
|\leq Ck \sqrt{j},
\end{equation}
for all $i=1,2,\ldots,\frac{n}{j}$.
This implies that the $l$-th component of the difference
\[
(\vec{R}((i+1)j)-\vec{R}(ij))- (\vec{C}((i+1)j)-\vec{C}(ij)))
\]
has a variance less or equal to $Ck \sqrt{j}$. But conditional
on $X=x$, we have that the differences \eqref{differences}
are normal, independent with expectation $\vec{0}$.
Taking the $l$-th component of 
\[
\vec{R}(t)-\vec{C}(t)
\]
at the times $j,2j,3j,\dots,n$
is thus like taking a standard Brownian motion   at the times
$t_1<t_2<\ldots <t_n$ where
\begin{align*}
&t_1=VAR[\vec{R}(j)-\vec{C}(j)],\\
&t_2=VAR[\vec{R}(2j)-\vec{C}(2j)]\\
&\ldots
\end{align*}
Since the variance increases at each step by at most $C k\sqrt{j}$,
we get that 
\[
t_{\frac{n}{j}}\leq Ck \sqrt{j}\cdot \frac{n}{j}=
Ck \cdot \frac{n}{\sqrt{j}}
\]
So, the maximum of the $l$-th component of the difference
$\vec{R}(t)-\vec{C}(t)$ restricted to the times $t=j,2j,3j,\ldots,\frac{n}{j}$
is thus stochastically bounded 
by the maximum of a standard Brownian motion during
the time interval 
\[
\left[0,Ck \cdot \left[\frac{n}{\sqrt{j}}\right]\right].
\]
The standard Brownian motion starting at the origin.
By the reflection principle for Brownian motion,
this distribution function can be bounded by
the distribution function of 
the standard Brownian motion at time
\[
T=Ck \cdot \left[\frac{n}{\sqrt{j}}\right].
\]
We have just shown that
conditional on $X=x$ and when $A^n$ holds,
then
\[
P(\max_{t=j,2j,3j,\dots,n}|R^l(t)-C^l(t)|\geq \log^2(n)\sqrt{T})
\leq 2P(\mathcal{N}(0,1)\geq \log^2(n))\leq  e^{-\log^{2}(n)/2},
\]
where the last inequality was obtained by using
the lemma \ref{boundstnormal} below.
We have just proven that
$$P(D^n |A^n)\geq 1-  e^{-\log^{2}(n)/2}.$$
\end{proof}

\begin{lemma}
We have that $P((F^{n})^{c})\leq e^{-\log^{2}(n)/3}$
for all $n$ large enough.
\end{lemma}

\begin{proof}
Let $F^n_{li}$ be the event that:
$$\sup_{s\in[0,j]}|C^l(ij+s)-C^l(ij)|\leq \log(n)\sqrt{j}$$
So, we have
\begin{equation}
\label{FNC}
P(F^{nc}_{li})=
P(\max_{t\in[0,j]}B_t\geq \log(n)\sqrt{j}
\end{equation}
where $B_t$ designates a standard Brownian motion starting
at the origin in a one dimensional space. By the reflection
principle, we get that the right side of \ref{FNC}
is equal to
\begin{equation}
\label{groslolo}P(F^{nc}_{li})=P(B_j\geq \log(n)\sqrt{j})=
P(\mathcal{N}(0,1)\cdot\sqrt{j}\geq \log(n)\sqrt{j})=
P(\mathcal{N}(0,1)\geq \log(n))
\end{equation}
where $\mathcal{N}(0,1)$ designates a standard normal.
Using Lemma \ref{boundstnormal}, we find
that the right side of \ref{groslolo} can be bounded
by $e^{-\log^{2}(n)/2}/2$ so that
\[
P((F^{n}_{li})^{c})\leq e^{-\log^{2}(n)/2}/2.
\]
Now
$$F^n=\cap_{li}F^n_{li}$$
where in the last interesection above $l$ ranges over $1,2,\ldots,k+1$
and $i$ ranges over $1,2,\ldots,n/j$. Hence,
we get
\[
P((F^{n})^{c})\leq \sum_{li}P(F^{nc}_{li})\leq 
\sum_{li}e^{-\log^{2}(n)/2}=\frac{n(k+1)}{j}e^{-\log^{2}(n)/2}
\]
where the dominant term is $e^{-\log^{2}(n)/2}$ which beats any power
of $n$. Note that $\frac{n(k+1)}{j}$ is just a power of $n$.
\end{proof}

\begin{lemma}
We have that $P((G^{n})^{c})\leq e^{-\log^{2}(n)/3}$
for all $n$ large enough.
\end{lemma}     
\begin{proof}Same as for $F^n$. So we leave it to the reader.

\end{proof}

\begin{lemma}
\label{boundstnormal}Let $\mathcal{N}(0,1)$ be a standard normal.
Let $s>0$.
Then, we have
$$P(\mathcal{N}(0,1)\geq s)\leq \frac{1}{2}\cdot\exp(-s^2/2).$$
\end{lemma}

\section{Simulations}\label{s:6}

In the pictures below we have simulated the optimal alignment with length $n$ running from $1$ to $10000$ and each time we compute the estimated standard deviation (of 100 simulations from a $0,1$ sequence).  The gaps are allowed only in one sequence and the number of gaps is restricted to be either a fixed number, a number $k=n^{\alpha}$ or a proportional number to $n$. 
\begin{figure}
 \centering
    \includegraphics[scale=.6]{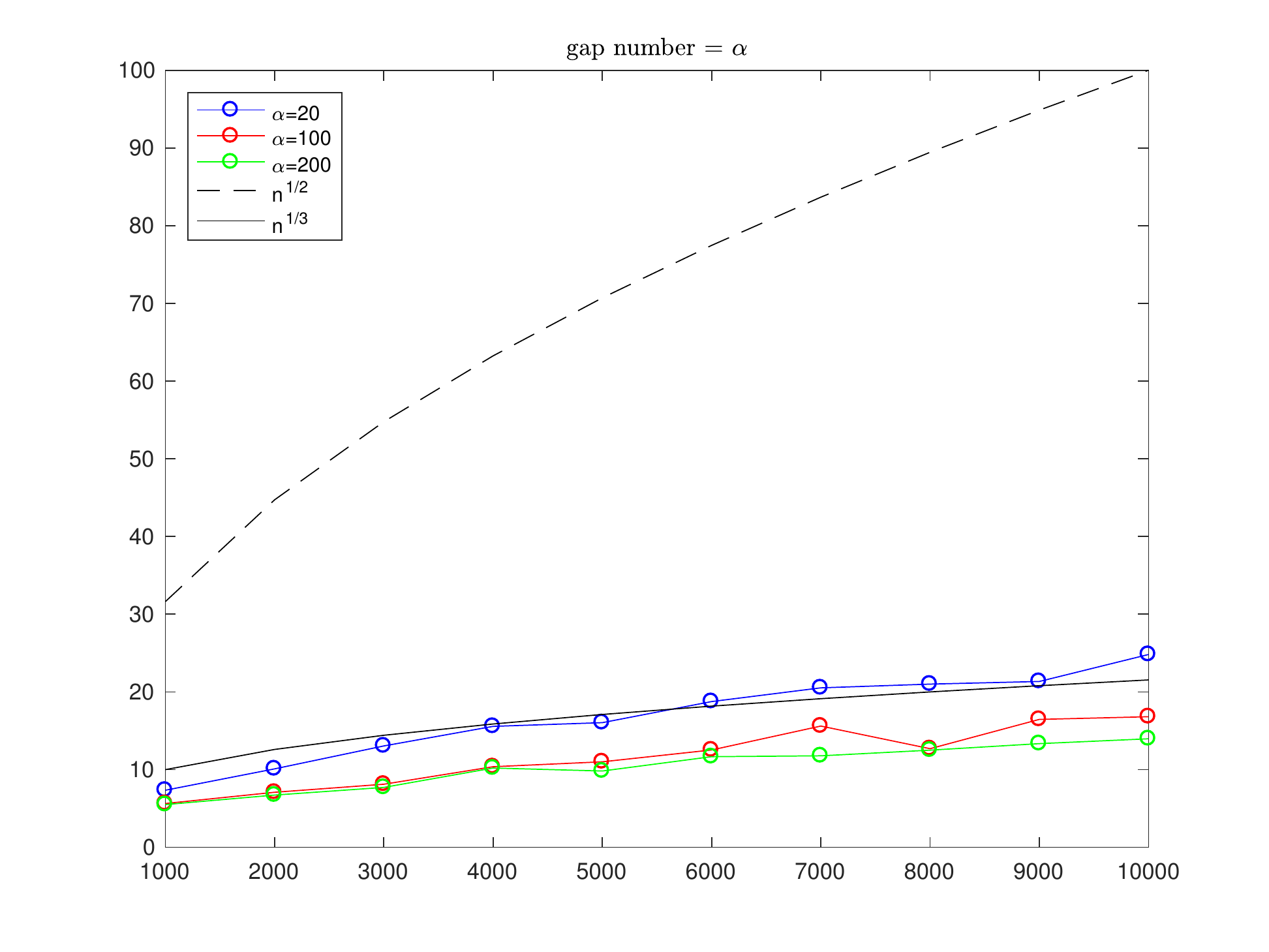} 
 \caption{ The horizontal line is the length of the sequences and the vertical line represents the standard deviation size of the optimal alignment.   The number of gaps in this simulation is 20, 100 and 200.  The dashed line is the plot of $n^{1/2}$ line and the solid black curve is $n^{1/3}$.   As it is clear from this picture, the standard deviation of the optimal alignment is suggested to be of size $n^{1/3}$  and not $n^{1/2}$.  }
\end{figure}

\begin{figure}[H]
   \centering
    \includegraphics[scale=.6]{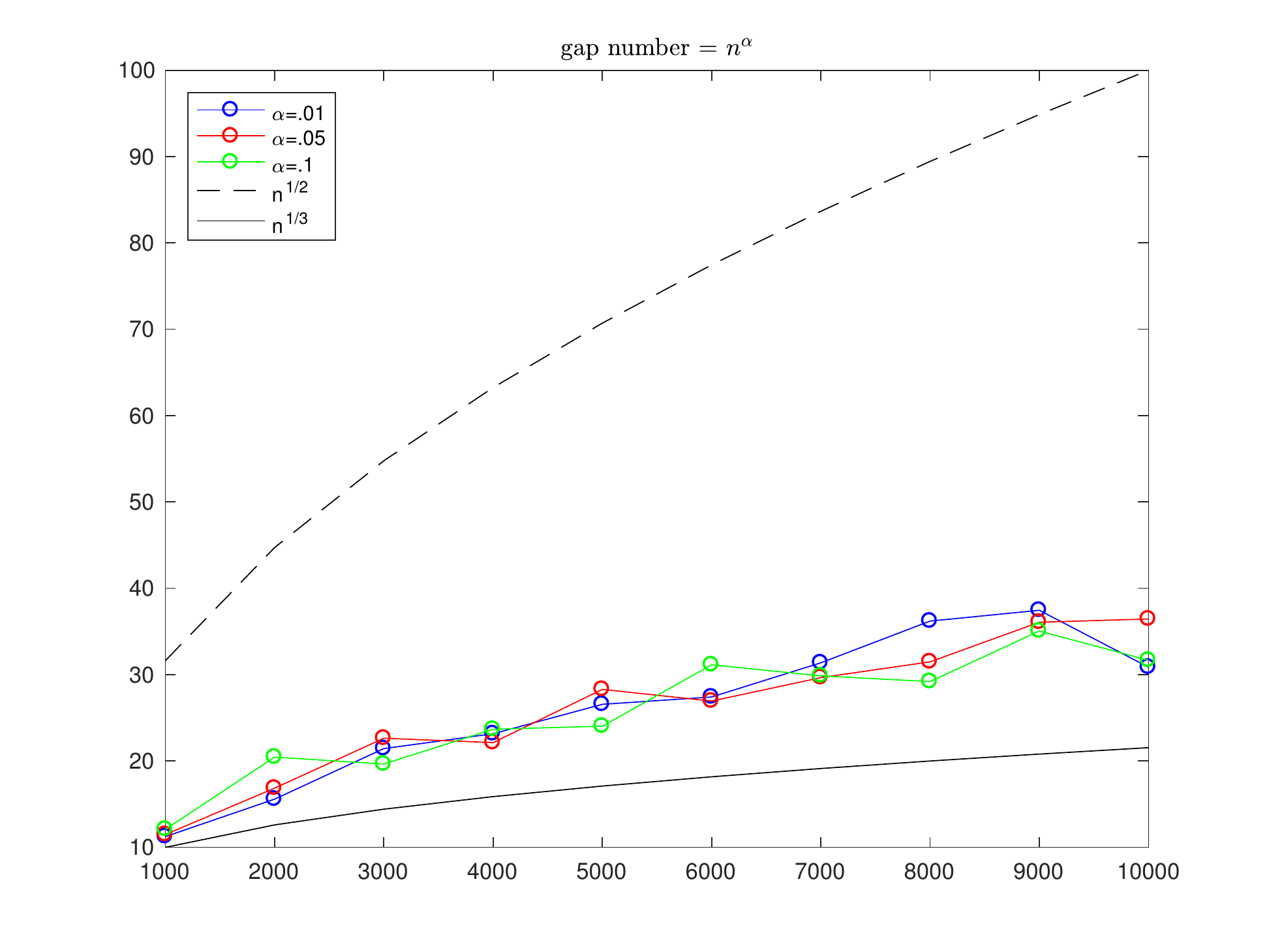} 
   \caption{ The number of gaps in this simulation is $n^{\alpha}$ for $\alpha=.1,.05,.01$.  As it is clear from this picture, the standard deviation of the optimal alignment is suggested to be of size $n^{1/3}$  and not $n^{1/2}$.  }
\end{figure}

\begin{figure}[H]
   \centering
    \includegraphics[scale=.6]{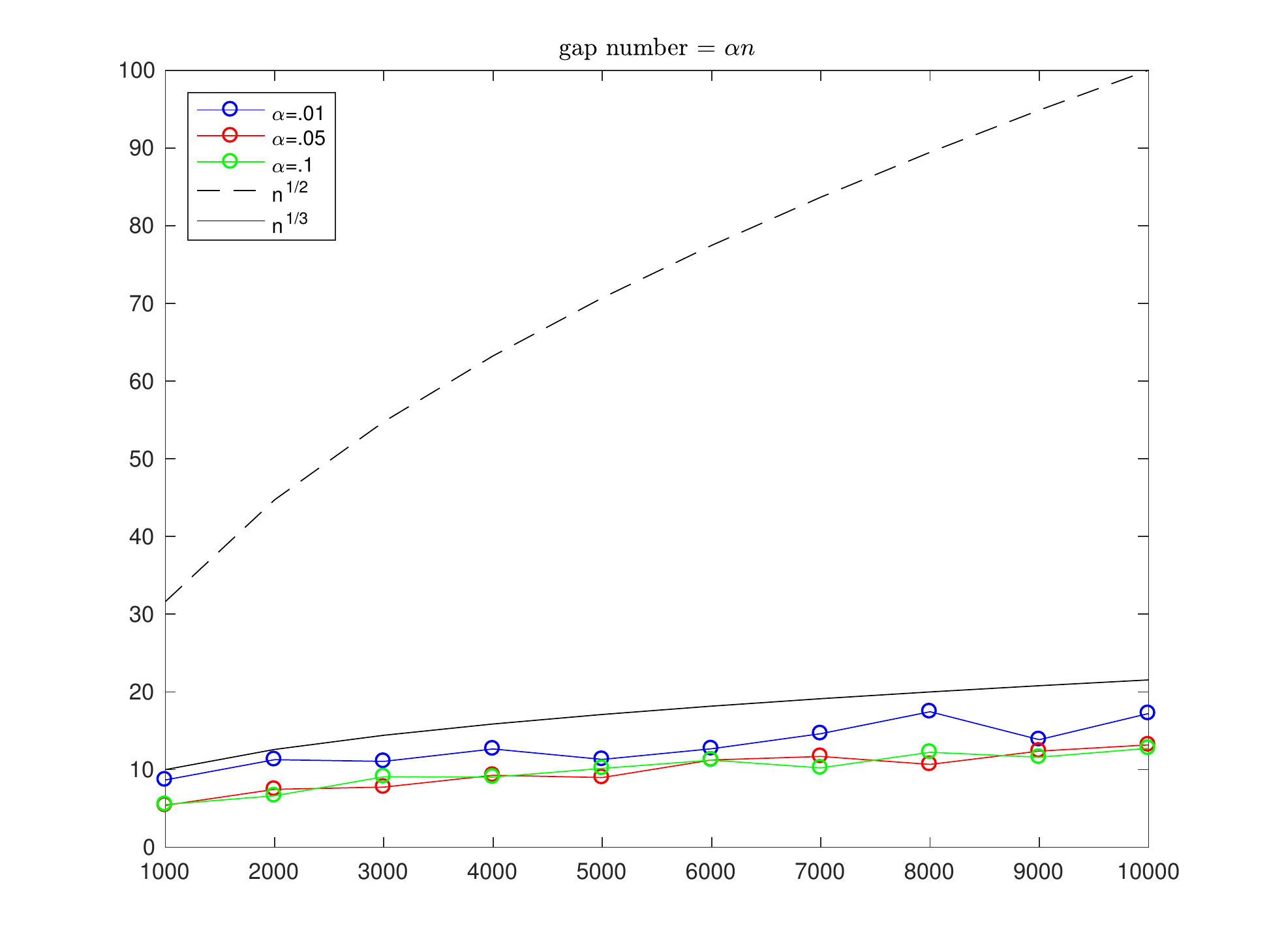} 
   \caption{ The number of gaps in this simulation is $\alpha n$ for $\alpha=.1,.05,.01$.     As it is clear from this picture, the standard deviation of the optimal alignment is suggested to be of size $n^{1/3}$  and not $n^{1/2}$.  }
\end{figure}

\section{Acknowledgements} 

We want to thank the reviewers of this paper for the comments which improved the paper.

\end{document}